\documentclass{amsart}
\usepackage{latexsym,amsxtra,amscd,ifthen}
\usepackage{amsfonts}
\usepackage{pifont}
\usepackage{verbatim}
\usepackage{amsmath}
\usepackage{graphicx}
\usepackage{amsthm}
\usepackage{amssymb}
\usepackage{url,color}
\usepackage[arrow,matrix]{xy}
\usepackage{pict2e,epic,eepic,pstricks}
\usepackage{multirow} 

\usepackage{anysize}\marginsize{25mm}{25mm}{30mm}{38mm}
\allowdisplaybreaks[4]
\newtheorem{thm}{Theorem}[section]

\newtheorem{cor}[thm]{Corollary}

\newtheorem{lem}[thm]{Lemma}
\newtheorem{que}[thm]{Question}

\newtheorem{prop}[thm]{Proposition}
\newtheorem{prop-def}[thm]{Proposition-Definition}
\theoremstyle{definition}
\newtheorem{Def}[thm]{Definition}
\theoremstyle{remark}
\newtheorem{rmk}[thm]{\bf Remark}
\newtheorem{exm}[thm]{\bf Example}
\numberwithin{equation}{section}
\def\N{\mathbb{N}}

\def\S{\mathbb{S}}

\def\sgn{{\rm sgn}}
\def\gd{\operatorname{Gdeg}}
\def\mdeg{\operatorname{mdeg}}

\usepackage{bbm}
\usepackage{float}

\newcommand{\lr}[1]{\langle #1 \rangle}
\newcommand{\ucr}[1]{\underset{#1}{\circ}}

\newcommand{\up}[1]{{^{#1}{\it \!\Upsilon}}}
\newcommand{\upa}[1]{{^{#1}{\it \!\Upsilon}}}

\def\la{\lambda}
\def \ot{\otimes}

\def \uas{u\mathcal{A}ss}
\def \Lie{\mathcal{L}ie}

\def \I{\mathcal{I}}
\def \J{\mathcal{J}}

\def \ip{{\mathcal{P}}}
\def \iq{\mathcal {Q}}
\def \iu{\!{\it{\Upsilon}}}
\def \dim{\operatorname{dim}}
\def \Inv{\operatorname{Inv}}
\def\gkdim{\operatorname{GKdim}}

\def\iend{\mathcal{E}{\rm nd}}

\def \Id{\operatorname{Id}}
\def \ker{\operatorname{Ker}}

\def \1{\mathbbm 1}

\def \k{\Bbbk}

\def \N{\mathbb{N}}
\def \S{\mathbb{S}}

\def\T{\mathcal{T}}
\def\sl{\operatorname{sl}}

\begin{document}
\title{Codimension sequence, grade, and generating degree:\\
an operadic approach}

\author{Y.-H. Bao}
\address{(Bao) School of Mathematical Sciences, Anhui University, Hefei 230601, China}
\email{baoyh@ahu.edu.cn}

\author{D.-X. Fu}
\address{(Fu) School of Mathematics and Statistics, Suzhou University, Suzhou 234000, China}
\email{dxfu@ahszu.edu.cn}

\author{J.-N. Xu}
\address{(Xu) School of Mathematical Sciences, Anhui University, Hefei 230601, China}
\email{j-n-xu@stu.ahu.edu.cn}

\author{Y. Ye}
\address{(Ye) School of Mathematical Sciences,
	University of Science and Technology of China, Hefei 230026, China}
\email{yeyu@ustc.edu.cn}

\author{J.J. Zhang}
\address{(Zhang) Department of Mathematics, Box 354350,
University of Washington, Seattle, Washington 98195, USA}
\email{zhang@math.washington.edu}

\author{Y.-F. Zhang}
\address{(Zhang)  School of Mathematical Sciences, Anhui University, Hefei 230601, China}
\email{yfzhang@stu.ahu.edu.cn}

\author{Z.-B. Zhao}
\address{(Zhao) School of Mathematical Sciences, Anhui University, Hefei 230601, China}
\email{zbzhao@ahu.edu.cn}




\begin{abstract} 
We study several classes of operadic ideals of the unital associative algebra operad $\uas$. 
As an application, we classify quotient operads of $\uas$ of GK-dimension $\leq 6$.
This corresponds to a classification of all T-ideals of codimension growth $n^g$ with $g\leq 5$
 (or equivalently, varieties of grade $g$ with $g\leq 5$).
\end{abstract}

\maketitle

\setcounter{section}{-1}
\section{Introduction}
\label{yysec0}

\subsection{PI-algebras}
\label{yysec0.1}
Throughout let $\Bbbk$ be a base field of characteristic zero. 
In this paper an algebra means an associative algebra with unit (=identity element). 
A \textit{polynomial identity} of an algebra $A$ is a noncommuting polynomial $f(x_1, \cdots, x_n)$ 
such that $f(a_1, \cdots, a_n)=0$ for all $a_1, \cdots, a_n\in A$. 
An algebra satisfying a nontrivial polynomial identity is called a \textit{PI-algebra}. 
The PI-theory originated from the work of Dehn \cite{De} and Wagner \cite{Wa}. 
Many significant developments of the PI-theory and PI-algebras can be found in \cite{GZ6, KR, Ka1, Ja, Pr, Ro}. 
Some newer results were recorded in \cite{AGPR, DF}.

\subsection{Connection to the operad $\uas$}
\label{yysec0.2}
It is well-known that the PI-theory such as the study of multilinear polynomial identities is related to 
the ymmetric operad $\uas$ which encodes the unital associative algebras. 
A PI-algebra is equivalent to an algebra over a quotient operad $\uas/\I$ for some nonzero operadic ideal $\I$ \cite[Lemma 2.1]{BXYZZ}. 
Recall that $\uas(n)$ is isomorphic to the right regular module over $\Bbbk\S_n$ for all $n\geq 0$, 
where $\S_n$ is the symmetric group of degree $n$. 
Let $V_n$ be the space consisting of all multilinear polynomials in $n$ variables $x_1, \cdots, x_n$. 
Clearly, $V_n$ admits a right action of $\S_n$, 
which is naturally isomorphic to the regular representation of $\S_n$ (or the $\S_n$-module $\uas(n)$). 
Let $A$ be a PI-algebra and $V_n(A)$ the subspace of $V_n$ of those polynomials that are identities of $A$. 
Then we obtain that
\begin{align*}
V_n(A)\cong \I_A(n)
\end{align*}
where $\I_A$ is the operad ideal of $\uas$ determined by the algebra $A$ \cite[Proposition 2.3]{BXYZZ}. 
As a consequence,
\begin{equation}
\label{E0.0.1}\tag{E0.0.1}
V_n/V_n(A)\cong (\uas/\I_A)(n)
\end{equation}  
for the same operadic ideal $\I_A$ of $\uas$. 
In this situation, we also say that $A$ is a PI-algebra associated to the operadic ideal $\I_A$. 
Conversely, for every nonzero operadic ideal $\I$ of $\uas$, 
a $\uas/\I$-algebra is a PI-algebra since each nonzero element in $\I$ gives an identity of $A$.   

\subsection{Codimension sequences/series}
\label{yysec0.3}
Recall that the {\it codimension sequence} of multilinear identities of a PI-algebra $A$ is defined to be 
\begin{align*}
c_n(A)\colon =\dim (V_n/V_n(A)), \quad \forall \; n\geq 0,
\end{align*}
which gives a measurement of the polynomial identities vanishing in the PI-algebra $A$. 
The codimension sequences are considered as one of most important invariants in PI-theory. 
It follows from \eqref{E0.0.1} that $c_n(A)=\dim (\uas/\I_A)(n)$ for all $n$. 
The associated {\it codimension series} of $A$ is defined to be 
\begin{align*}
c(A,t)=\sum_{n\geq 0} c_n(A)t^n.
\end{align*}

One main object of this paper is the codimension sequences/series of associative PI-algebras with unit. 
We will not consider algebras without unit (nor other classes of algebras), 
as their codimension sequences/series behave differently. 
The exact values of the codimension sequences are known for very few PI-algebras, 
among them the Grassmann algebra $E$, 
the matrix algebra $M_2(\Bbbk)$ and the tensor square $E\otimes E$ of the Grassmann algebra, 
the algebras of $k\times k$ upper triangular matrices $U_k(\Bbbk)$ and $U_k(E)$ with entries from the field $\Bbbk$ 
and from the Grassmann algebra $E$, respectively, see \cite[Examples 9-13]{BD}. 
In this paper, we provide more families of algebras with the codimension series being described explicitly. 

The codimension series have been studied extensively, see, for example, 
\cite{AGPR, AJK, Be, BR, BR2, BD, DrR, GMP, GMZ, GZ1, GZ2, GZ3, GZ4, GZ5, GZ6, IKM, Ma} and references therein. 
In \cite{Re} Regev conjectured that, for every PI-algebra $A$, there are $g\in \frac{1}{2}{\mathbb Z}$, $p\in {\mathbb N}$, 
and $\lambda\in {\mathbb R}_{>0}$, such that
\begin{align}
\label{E0.0.2}\tag{E0.0.2}
c_n(A)\sim \lambda n^{g} p^n,\qquad  n\geq 0.
\end{align}
This conjecture was proved by Berele \cite{Be} and Berele-Regev \cite{BR}. 

\begin{Def}
\label{yydef0.1} 
Let $A$ be a PI-algebra.
\begin{enumerate}
\item[(1)]
The algebra $A$ or its codimension series $c(A,t)$ is called of {\it type $(g,p)$} with $g\in \frac{1}{2}{\mathbb Z}$ 
and $p\in {\mathbb N}$ if \eqref{E0.0.2} holds. 
In this case we say that $A$, or the associated T-ideal, 
or the associated operadic quotient of $\uas$, has {\it codimension growth} $n^g p^n$. 
\item[(2)] 
The integer $p$ in \eqref{E0.0.2} is called the {\it exponent} or {\it PI-exponent} of $A$ and denoted by $\exp(A)$. 
\item[(3)]
The number $g$ in \eqref{E0.0.2} is called the {\it $p$-grade} of $A$ and denoted by $g^p(A)$.
\item[(4)]
The $1$-grade of $A$ (when $p=1$) is simply called the {\it grade} of $A$ and denoted by $g(A)$.
\item[(5)]
For each fixed pair $(g,p)$, let $\Lambda_{(g,p)}$ be the set of real numbers $\lambda$ that appears in \eqref{E0.0.2}.
\end{enumerate} 
\end{Def}

We post the following question.

\begin{que}
\label{yyque0.2}
For every pair $(g,p)$ with $g\in \frac{1}{2}{\mathbb Z}$ and $p\in {\mathbb N}$, 
are there only finitely many distinct codimension series $c(A,t)$ of type $(g,p)$?
\end{que}

Note that there are infinitely many (and uncountably many if $\Bbbk$ is uncountable) operadic ideals of $\uas$ 
with given codimension series (of codimension growth $n^5$), 
see \cite[Section 5]{GMZ} and Corollary \ref{yycor4.7}.

\subsection{Some previously known results}
\label{yysec0.4}
The first positive evidence for Question \ref{yyque0.2} is the following well-known result.

\begin{prop}
\cite{BYZ, DrR, GMP}
\label{yypro0.3} 
Let $g\in \frac{1}{2}{\mathbb Z}$.
\begin{enumerate}
\item[(1)]
If $g$ is not an integer or $g<0$, then there is no codimension series of type $(g,1)$.
\item[(2)]
If $g$ is a nonnegative integer, then there are finitely many codimension series of type $(g,1)$ 
and a very rough upper bound is $\prod_{i=2}^{g+1} (i!-1)$.
\end{enumerate}
\end{prop}

An operadic version of Proposition \ref{yypro0.3} holds 
for the set of quotient operads of any given locally finite 2-unitary operads, 
see \cite[Theorem 5.3]{BYZ} and its proof. 

When the grade $g\leq 3$, here is a complete list of codimension series $c(A,t)$ of type $(g,1)$:
\begin{enumerate}
\item[(1)]
There is only one codimension series of type $(0,1)$, namely,
$$c(A,t)=\frac{1}{1-t}.$$
Consequently, $\Lambda_{(0,1)}=\{1\}$.
\item[(2)]
There is no codimension series of type $(1,1)$. 
Consequently, $\Lambda_{(1,1)}=\emptyset$.
\item[(3)]
There is only one codimension series of type $(2,1)$, namely,
$$c(A,t)=\frac{1}{1-t}+\frac{t^2}{(1-t)^3}.$$
Consequently, $\Lambda_{(2,1)}=\{\frac{1}{2}\}$.
\item[(4)]
There is only one codimension series of type $(3,1)$, namely,
$$c(A,t)=\frac{1}{1-t}+\frac{t^2}{(1-t)^3}+\frac{2t^3}{(1-t)^4}.$$
Consequently, $\Lambda_{(3,1)}=\{\frac{1}{3}\}$.
\end{enumerate}

\begin{thm}
\label{yythm0.4} 
When the grade $g=4$, there are exactly 10 distinct codimension 
series of type $(4,1)$ as given below.
\begin{enumerate}
\item[(a)]
$$c(A,t)=\frac{1}{1-t}+\frac{t^2}{(1-t)^3}+
\frac{t^4}{(1-t)^5}$$
or
\[c_n(A)=1+{n\choose 2}+{n\choose 4}.\]
\item[(b)]
$$c(A,t)=\frac{1}{1-t}+\frac{t^2}{(1-t)^3}+
\frac{2t^3}{(1-t)^4}+\frac{ut^4}{(1-t)^5}$$
or
\[c_n(A)=1+{n\choose 2}+2{n\choose 3}+u{n\choose 4}\]
where $u$ is an integer in the set $\{1,2,3,4,5,6,7,8,9\}$.
\end{enumerate}
As a consequence, 
$\Lambda_{(4,1)}=\{\frac{u}{4!}\mid 1\leq u\leq 9, u\in \N\}$.
\end{thm}

The above results can also be found in \cite[Section 5]{OV}.

These facts are well-known to many experts, see, for example, 
\cite{BYZ, DrR, GMP, GMZ, Ma, OV}. 
The second evidence for a 
positive answer to Question \ref{yyque0.2} is the following.

\begin{prop}
\label{xxpro0.5}
The set
$$\bigcup_{\Bbbk} \{ c(A,t)\mid {\text{all PI-algebras $A$ 
over $\Bbbk$}}\}
$$
is countable where the union $\bigcup_{\Bbbk}$ runs over 
all fields $\Bbbk$ of characteristic zero.
\end{prop}

The third evidence for Question \ref{yyque0.2} is a result 
of Giambruno-Zaicev \cite{GZ7, GZ8} which states that there 
are only finitely many minimal varieties respect to the 
PI-exponent. Other evidences for Question \ref{yyque0.2} 
and related results can be found in \cite{Ma, Ma2}.

\subsection{Result I: GK-dimension 5 or grade 4}
\label{yysec0.5}
Our approach is to study the quotient operads of the operad
$\uas$. In \cite[Theorem 0.7]{BYZ}, the authors classified 
all quotient operads of $\uas$ of GK-dimension $4$ or less. 
Let $\uas/\I$ be a quotient operad of GK-dimension $d$. If 
$d=1$, $3$ or $4$, then $\I$ is just the $d$-th truncation 
ideal $\up{d}$ of $\uas$ \eqref{E1.0.7}. And the GK-dimension 
of $\uas/\I$ can not be $2$. Therefore, there is a unique
codimension sequence $\{c_n(A)\}$ (or equivalently, unique
codimension series $c(A,t)$) of multilinear identities 
when the grade of $A$ is either $0$, $2$ or $3$,
see the list given before Theorem \ref{yythm0.4} .

In \cite{GMP} authors said that ``{\it Unfortunately, a 
classification of the T-ideals $Id(A)$ for which $c_n(A) 
\cong q n^k$, $k \geq 4$ seems to be out of reach at present}'', 
see the paragraph after \cite[Theorem 3.6]{GMP}. 
On codimension series of type $(4, 1)$, Giambruno-La Mattina-Zeicev 
gave a classification of the minimal varieties of polynomial 
growth $n^4$ in \cite{GMZ}, which is equivalent to classify 
maximal operadic ideals $\I$ with $\gkdim(\uas/\I)=5$. Recall that 
a variety $\mathcal{V}$ is said to be {\it minimal} of polynomial 
growth $n^k$ if any proper subvariety of $\mathcal{V}$ has polynomial 
growth $n^t$ with $t<k$. Equivalently, we say that an operadic
ideal $\I$ of $\uas$ is {\it maximal} with respect to $\gkdim$ if 
$\gkdim(\uas/\I)=k$ and  $\gkdim \uas/\mathcal{J}<k$ for any 
operadic ideal $\mathcal{J}\supsetneq \I$. Based on 
\cite{GMZ}, Oliveira and Vieira obtained a complete list of 
varieties generated by unital algebras with growth 
$n^4$ and computed their codimension sequences in \cite{OV}.

One project of this paper is to find operadic ideals $\I$ 
of $\uas$ such that $\uas/\I$ is of GK-dimension 5. 
Here is the operadic version of the classification of varieties 
of grade 4 in \cite{OV}.

\begin{thm}
\label{yythm0.6}
Let $\I$ be an ideal of $\uas$ such that the GK-dimension of the quotient operad $\ip:=\uas/\I$ is $5$. 
Then either {\rm{(i)}} $\I=\up{4}^M$ {\rm{[Definition \ref{yydef4.1}]}} where $M$ is one of 15 proper submodules of $\up{4}(4)$ 
or {\rm{(ii)}} $\I=\upa{4}^{\mathbb M}$ {\rm{[Definition \ref{yydef4.10}]}} 
where ${\mathbb M}$ is the unique $4$-admissible pair $(\up{3}(3), \up{3}(4)\cap \up{4}(4))$. 
\end{thm}

In this paper we propose an operadic approach to the 
classification project of varieties of grade larger than 
4. The main ingredient is the notation of a generalized 
truncation ideal of $\uas$ associated to an admissible 
sequence. In fact, there is a one-to-one correspondence 
between the set of generalized truncation ideals 
$\I \subseteq \uas$ and the set of admissible sequences 
[Theorem \ref{yythm4.16}]. Some basic properties of 
generalized truncation ideals are proven in Section 
\ref{yysec4}. Some classification results (see
Theorems \ref{yythm0.6} and \ref{yythm0.10}) are only 
dependent on generalized truncation ideals of types I
and II [Definitions \ref{yydef4.1} and \ref{yydef4.10}].

\subsection{Generating degree and generating element}
\label{yysec0.6}
By \cite[Theorem 0.3(2)]{BXYZZ}, every proper ideal of 
$\uas$ is generated by one element. Hence the following
definition was proposed.

\begin{Def}\cite[Definition 0.4]{BXYZZ}
\label{yydef0.7}
Let $\I$ be an ideal of $\uas$. Suppose that 
$\theta=\sum_{\sigma\in \S_n}c_\sigma \sigma\in \I(n)$ 
with the least $n$ such that $\I=\lr{\theta}$. Then 
$f_\theta:=\sum_{\sigma\in \S_n}c_\sigma 
x_{\sigma^{-1}(1)}\cdots x_{\sigma^{-1}(n)}$ is called a 
\textit{generating element} of $\I$, and $n$ is called 
the \textit{generating degree} of $\I$ and denoted by
$\gd(\I)$.

Let $A$ be a PI-algebra and $\I_A$ be the (operadic) ideal 
of $\uas$ associated to $A$. Suppose that 
$\theta=\sum_{\sigma\in \S_n}c_\sigma \sigma\in \I_A(n)$ 
with the least $n$ such that $\I_A=\lr{\theta}$. Then the 
polynomial 
$$f_\theta=\sum_{\sigma\in \S_n}c_\sigma 
x_{\sigma^{-1}(1)}\cdots x_{\sigma^{-1}(n)}$$ is called a 
\textit{generating identity} of $A$, and $n$ is called 
the \textit{generating degree} of $A$ and denoted by
$\gd(A)$. 
\end{Def}

Observe that the generating degree may not equal to the 
minimal degree, denoted by $\mdeg$, of a PI-algebra 
introduced in \cite{AL} (also see Definition \ref{yydef1.11}). 
For example, the third truncation ideal 
$\up{3}=\lr{\up{3}(4)}\neq \lr{\up{3}(3)}$, so
$\mdeg(\up{3})=3<4=\gd(\up{3})$. Note that $\gd-\mdeg$
can be arbitrary large, see Remark \ref{yyrem1.1}.
By definition, the generating 
degree of $A$ is uniquely defined, though the generating 
identity is not unique. 

\subsection{Result II: an upper bound of $\gd$ by 
$\gkdim$}
\label{yysec0.7}
Another main object of this paper is the generating degree 
(and generating identities). We have the following result.

\begin{thm}
\label{yythm0.8}
Let $\I$ be an ideal of $\uas$. Then 
$$\gd(\I)\leq \begin{cases}
\gkdim (\uas/\I)+1 & {\text{ if $\quad$ $\gkdim (\uas/\I)$ is odd;}}\\
\gkdim (\uas/\I)  & {\text{ if $\quad$  $\gkdim (\uas/\I)$ is even.}}
\end{cases}
$$
\end{thm}

We also calculate explicitly the generating degree of 
$\up{k}$ for all $k\geq 2$ [Proposition \ref{yypro3.10}]
and the generating degree of $A$ when $A$ has grade at most 4
in Theorem \ref{yythm0.9}.

\subsection{Result III: Classification of $\gkdim\leq 6$
or equivalently grade $\leq 5$}
\label{yysec0.8}

\begin{thm}
\label{yythm0.9}
Let $A$ be a PI-algebra of grade $g$. Suppose  
$\I_A$ is the operadic ideal associated to $A$.
\begin{enumerate}
\item[(1)] 
If $g=0$, then $\I_A=\up{1}$, $\gd(A)=2$, and 
\[f(x_1, x_2)=[x_1, x_2]\]
is a generating identity. 
\item[(2)]	
The grade of $A$ can not be 1.
\item[(3)] 
If $g=2$, then $\I_A=\up{3}$, $\gd(A)=4$, and 
\begin{align*}
	f(x_1, x_2, x_3, x_4)
	=[x_2, x_1][x_4, x_3]+[x_4, x_3, x_2, x_1]+[x_3, x_2, x_1]x_4
\end{align*}
is a generating identity of $A$.
\item[(4)] 
If $g=3$, then $\I_A=\up{4}$, $\gd(A)=4$, and 
\[f(x_1, x_2, x_3, x_4)
=[x_2, x_1][x_4, x_3]+[x_4, x_3, x_2, x_1]\]
is a generating identity of $A$.
\item[(5)] 
If $g=4$, then the associated operadic ideals of $A$ 
is one of 16 cases given in Theorem \ref{yythm0.6}.
\begin{enumerate}
\item[(a)] 
$\I_A=\up{4}^M$ with $M$ a proper submodule of $\up{4}(4)$. 
Then the generating degrees and the generating 
identities are listed in Table \eqref{E5.7.2}. 
{\rm{(}}There are 15 cases.{\rm{)}}
\item[(b)] 
$\I_A=\upa{4}^{\mathbb M}=\lr{\up{3}(3)}+\up{5}$. In this case, 
$\gd(A)=6$, and 
\[f=[x_3, x_1, x_2]x_4x_5x_6+\sum_{\sigma\in \S_6}
\sgn(\sigma)x_{\sigma(1)}\cdots x_{\sigma(6)}\]
is a generating identity of $A$.
\end{enumerate}
\end{enumerate}
\end{thm}

Finally we have an operadic version of a partial classification of 
the generating degree of $A$ when $A$ has grade 5 (or equivalently, 
$\gkdim (\uas/\I_{A})=6$).

\begin{thm}
\label{yythm0.10}
Let $\I$ be an ideal of $\uas$ such that the 
GK-dimension of the quotient operad $\ip:=\uas/\I$ is $6$. 
Then either {\rm{(i)}} $\I=\up{5}^M$ 
{\rm{[Definition \ref{yydef4.1}]}} where $M$ is any
proper submodule of $\up{5}(5)$ or {\rm{(ii)}} 
$\I=\upa{5}^{\mathbb M}$ {\rm{[Definition \ref{yydef4.10}]}} 
where ${\mathbb M}$ is any $5$-admissible pair $(M_4,M_5)$. 
A partial classification of $5$-admissible pairs is given 
in Lemma {\rm{\ref{yylem5.10}}} and Remark {\rm{\ref{yyrem5.14}}}.
\end{thm}

By the proof of Corollary \ref{yycor4.7}, there are 
infinitely many submodules $M\subset \up{5}(5)$. For 
each case its codimension series can be determined
[Corollary \ref{yycor4.6}(1,2)]. By Lemma \ref{yylem5.10}, 
there are infinitely many $5$-admissible pairs $(M_4,M_5)$ 
and its codimension series can be determined, see Lemma 
\ref{yylem4.14}.

\begin{rmk}
\label{yyrem0.11}
The classification given in Lemma \ref{yylem5.10} is not 
complete, especially for the infinite family in Lemma 
\ref{yylem5.10}(4). On the other hand, a complete list of 
codimension series is given in Theorem \ref{yythm5.15}.
The generating degree is also determined except for a 
few cases in the infinite family in Lemma 
\ref{yylem5.10}(4), see Lemma \ref{yylem5.11} and 
Table \eqref{E5.14.1}. Most of these computations are based on computer algebra, and 
non-essential details are skipped in order to save space. 
For a classification project of varieties of higher grades,
Theorem \ref{yythm4.16}(3) will be useful. 
\end{rmk}

\subsection{Organization}
\label{yysec0.9}
This paper is organized as follows. We collect some 
basic notions, notation and facts about operads in 
Section 1. In Section 2, we give some relations among 
PI-algebras, T-ideals and (operadic) ideals of $\uas$. 
In Section 3, we deal with the truncation ideals of $\uas$
and prove Theorem \ref{yythm0.8}. In Section 4, we study 
generalized truncation ideals of 2-unitary operads $\ip$. 
In section 5, we give a classification of (operadic) 
ideals of $\uas$ of GK-dimension 5 and 6, and study the 
related invariants of PI-algebra of low grades.

\section{Preliminaries}
\label{yysec1}

Throughout this paper, $\k$ is a fixed field of characteristic 
zero, and operads are $\Bbbk$-linear and symmetric. All 
unadorned $\ot$ will be $\ot_\k$. For convenience, we denote 
$[n]=\{1, 2, \cdots, n\}$, and $\lambda\vdash n$ means that 
$\la$ is a partition of $n$. In this section, we recall some 
notation and basic facts about operads and, in particular, the 
unital associative algebra operad $\uas$. We refer \cite{BYZ} 
for more details. 

\subsection{The unital associative algebra operad $\uas$}
\label{yysec1.1}
Let $\S_n$ be the symmetric group of degree $n$. We follow 
that convention in \cite{BYZ} and use the sequence 
$(\sigma^{-1}(1), \sigma^{-1}(2), \cdots, \sigma^{-1}(n))$ 
to denote an element $\sigma\in \S_n$. Equivalently, each 
$(i_1, i_2, \cdots, i_n)$ of $[n]$ corresponds to the 
permutation $\sigma\in \S_n$ given by $\sigma(i_k)=k$ 
for all $1\le k\le n$. We also use $1_n$ to denote the 
identity element in $\S_n$.

Let $\uas$ be the operad encoding unital associative algebras.
Recall that $\uas(n)\cong \Bbbk\S_n$ is the right regular 
$\Bbbk\S_n$-module for all $n\geq 0$, and the composition map 
of $\uas$ is linearly extended by the following maps: for $n>0$, 
$k_1, k_2, \cdots, k_n\ge 0$,
\begin{align*}
\S(n) \times \S(k_1)\times \cdots \times \S(k_n)
   &\to \S(\sum_{i=1}^n k_i),\\
(\sigma, \sigma_1, \cdots, \sigma_n) 
  &\mapsto (\tilde{B}_{\sigma^{-1}(1)}, \cdots, 
	\tilde{B}_{\sigma^{-1}(n)})
\end{align*}
for all $\sigma\in \S_n$ and $\sigma_i\in \S_{k_i}$, 
$1\le i\le n$, where 
\[\tilde{B}_{i}=(\sum_{j=1}^{i-1}k_j+\sigma_i^{-1}(1), 
\cdots, \sum_{j=1}^{i-1}k_j+\sigma_i^{-1}(k_i))\]
for all $i=1, \cdots, n$. The partial composition 
\[\uas(m) \ucr{i} \uas(n)  \to \uas(m+n-1)\]
is given by 
\[\mu\ucr{i}  \nu=\mu\circ (1_1, \cdots, 
\underset{i}{\nu}, \cdots, 1_1)\]
for $\mu\in \uas(m), \nu\in \uas(n)$, $m\ge 1, n\ge 0$ 
and $1\le i\le m$.

The operad $\uas$ encodes unital associative algebras, namely, a unital associative algebra is exactly a $\uas$-algebra. 
Let $(A, \mu, u)$ be a unital associative algebra. One can define an operad morphism 
$\gamma=(\gamma_n)\colon \uas \to \mathcal{E}nd_A$
given by $\gamma_0(1_0)=u$ and $\gamma_2(1_2)=\mu$, 
where $\mathcal{E}nd_A$ is the endomorphism operad of the vector space $A$, see \cite[Section 5.2.11]{LV}.
Each $\theta=\sum_{\sigma\in \S_n} c_\sigma \sigma\in \Bbbk\S_n$ gives an $n$-ary operation on $A$, 
\begin{align}\notag
\gamma_n(\theta)\colon A^{\otimes n} \to A, \quad 
\gamma_n(\theta)(r_1\otimes\cdots\otimes r_n)
=\sum_{\sigma\in \S_n} c_\sigma r_{\sigma^{-1}(1)}
\cdots r_{\sigma^{-1}(n)}.
\end{align}

\subsection{The Lie operad $\Lie$}
\label{yysec1.2}
Let $\Lie$ be the Lie operad that encodes the category 
of Lie algebras. It is well known that $\Lie(0)=0$ and 
$\Lie(n)$ can be viewed as a $\Bbbk\S_n$-submodule of 
$\uas(n)\cong\Bbbk\S_n$ and $\dim\Lie(n)=(n-1)!$ for each 
$n\ge 1$. Applying the $\S_n$-module isomorphism between 
$\Lie(n)$ and the multilinear part of 
$\Lie(\bigoplus_{i=1}^n\Bbbk x_i)$, one can construct a 
basis of $\Lie(n)$ by the Dynkin elements
\begin{align}
\label{E1.0.1}\tag{E1.0.1}
[[\cdots[[x_n, x_{\sigma^{-1}(1)}], x_{\sigma^{-1}(2)}], 
\cdots], x_{\sigma^{-1}(n-1)}]
\end{align}
for all $\sigma\in \S_{n-1}$, see \cite[Lemma 5.1.1 and 
Section 13.2.4]{LV} or \cite{Reu}. In fact, we denote 
$\tau=(1)-(12)\in \Bbbk\S_2$ and 
\begin{align}\label{E1.0.2}\tag{E1.0.2}
\tau_n=\underbrace{\tau\ucr{1}
(\cdots (\tau\ucr{1}\tau))}_{n-1\, 
{\footnotesize \rm copies}}\quad (n\ge 2).
\end{align}
Then the corresponding element of \eqref{E1.0.1} is 
\[\tau_n\ast (n, \sigma^{-1}(1), \cdots, \sigma^{-1}(n-1))\]
with $\sigma\in \S_{n-1}$. Therefore, we obtain a basis 
$\{\tau_n\ast (n, \sigma^{-1}(1), \cdots, \sigma^{-1}(n-1))
\mid \sigma\in \S_{n-1}\}$ of $\Lie(n)$ for each $n\ge 2$.

%
%

\subsection{Exponent and GK-dimension of an operad}
\label{yysec1.3}
Let $\ip=(\ip(n))$ be an operad. Assume that $\ip$ is 
locally finite,  i.e. each $\ip(n)$ is finite dimensional 
for all $n\in \N$. Recall that the \textit{exponent} of 
$\ip$ is defined to be 
\[\exp(\ip)\colon =\limsup_{n\to \infty} \sqrt[n]{\dim \ip(n)}.\]
An operad is said to \textit{have exponential growth} if 
$\exp(\ip)>1$. We say that an operad \textit{has polynomial 
growth} if there exists constants $C, k>0$ such that $\dim 
\ip(n)<Cn^k$ for all $n\in \N_+$. The \textit{Gelfand-Kirillov 
dimension} (or \textit{GK-dimension} for short) of $\ip$ is 
defined to be 
\[\gkdim(\ip)\colon=\limsup_{n\to \infty} 
\log_n(\sum_{i=0}^n \dim \ip(i)),\]
and the \textit{generating series} of $\ip$ is defined to be 
\[G_\ip(t)=\sum_{n=0}^\infty (\dim\ip(n))t^n.\]

Recall from Definition \ref{yydef0.1}(2) that the 
\textit{PI-exponent} of a PI-algebra $A$ is defined to be
\begin{equation}
\label{E1.0.3}\tag{E1.0.3}
\exp(A)=\lim_{n\to \infty}\sqrt[n]{c_n(A)}
\end{equation}
(also denoted by $\Inv(A)$) \cite[Definition, p.222]{GZ3}. 
One very nice result of Giambruno and Zaicev states that 
$\exp(A)$ always exists and is an integer \cite[Theorem 1]{GZ3}. 
If $\ip=\uas/\I_A$, then $\exp(\ip)=\exp(A)$. The 
\textit{grade} of a PI-algebra was introduced in 
Definition \ref{yydef0.1}(4). Clearly, the grade of $A$ is 
exactly 1 less than the GK-dimension of the corresponding 
quotient operad $\uas/\I_A$. 

\subsection{Some operators on unitary operads}
\label{yysec1.4}
Recall that an operad $\ip$ is said to be \textit{unitary} 
if $\ip(0)=\Bbbk \1_0$ with a basis element $\1_0$ (called a 
\textit{$0$-unit}), and a unitary operad is said to be 
\textit{2-unitary} if there exists an element $\1_2\in \ip(2)$ 
(called a \textit{$2$-unit}) such that 
\begin{align}
\label{E1.0.4}\tag{E1.0.4}
\1_2\circ(\1_1, \1_0)=\1_1=\1_2\circ(\1_0, \1_1),
\end{align}
where $\1_1\in \ip(1)$ is the identity of $\ip$. If in addition, 
\begin{align}
\label{E1.0.5}\tag{E1.0.5}
\1_2\ucr{1}\1_2=\1_2 \ucr{2}\1_2
\end{align} we say that $\ip$ is \textit{2a-unitary}. 
Clearly, $\uas$ is a 2a-unitary operad. We refer to 
\cite[Section 2.3]{BYZ} for more details. A non-symmetric 
operad with a 2-unit satisfying \eqref{E1.0.4}-\eqref{E1.0.5}
is called an \textit{operad with a multiplication} in 
\cite{Me, Ko}, and also called a \textit{strict unital comp 
algebra} in \cite{GS}.

Let $\ip$ be a unitary operad and $I$ a subset of $[n]$.
Recall that a {\it restriction operator} 
\cite[Section 2.2.1]{Fr} means
\begin{equation}
\label{E1.0.6}\tag{E1.0.6}
\pi^{I}\colon \ip(n)\to\ip(s),\qquad \pi^{I}(\theta)=
\theta\circ(\1_{\delta_{I}(1)}, \cdots, \1_{\delta_{I}(n)})
\end{equation}
for all $\theta\in \ip(n)$, where $\delta_{I}$ is the 
characteristic function of $I$, i.e. $\delta_{I}(x)=1$ 
for $x\in I$ and $\delta_{I}(x)=0$ otherwise. If 
$I=\{i_1, \cdots, i_s\}\subset [n]$ with 
$1\le i_1<\cdots<i_s\le n$, we also denote
$\pi^I$ as $\pi^{i_1, \cdots, i_s}$. 

 
Let $\ip$ be a 2-unitary operad with a 2-unit $\1_2$. 
The \textit{extension operator} 
$\Delta_I\colon \ip(n)\to \ip(n+s)$
is defined by 
\[\Delta_I(\theta)=\theta\circ 
(\1_{\delta_I(1)+1}, \cdots, \1_{\delta_I(n)+1})\]
for all $\theta\in \ip(n)$, and the function 
$\iota_r^l\colon \ip(n)\to \ip(l+n+r)$ is defined by
\[\iota_r^l(\theta)\colon= \1_3\circ (\1_l, \theta, \1_r),\]
where $\1_n$ is defined by $\1_n=\1_2\circ (\1_{n-1}, \1_1)$ 
inductively. If $l=0$ or $r=0$, then we have 
$\iota^0_r(\theta)=\1_2\circ (\theta, \1_r)$ and 
$\iota_0^l(\theta)=\1_2\circ (\1_l, \theta)$, which are 
denoted by $\iota_r$ and $\iota^l$, respectively.

\subsection{Truncation ideals}
\label{yysec1.5}
We recall a key concept defined in \cite{BYZ} that will 
be used in later part of this paper. Let $\ip$ be a 
unitary operad. For each integer $k\ge0$, the 
{\it truncation ideal} $\up{k}_\ip$ is defined by
\begin{equation*}
\label{E1.0.7}\tag{E1.0.7}
^k\iu_\ip(n)=\begin{cases}
\bigcap\limits_{I\subset [n],\, |I|= k-1} \ker\pi^I, 
& \text{if }n\ge k;\\
\quad\ \ 0, & \text{otherwise}.
\end{cases}
\end{equation*}
If no confusion, we write ${^k\iu}={^k\iu_\ip}$ for 
brevity. 

\begin{rmk}
\label{yyrem1.1}
Truncation ideals are useful for constructing a decreasing
sequence of operadic ideals (and the corresponding T-ideals 
when we are working with $\uas$). Similar to Theorem 
\ref{yythm0.4}(a), one can construct a family of codimension 
series as follows: Let $\I_{E}$ be the ideal of 
$\uas$ associated to the Grassmann algebra $E$. It is known
that $\I_{E}=\lr{\up{3}(3)}$, which was denoted by $\T_3$ as 
in Theorem \ref{yythm0.6}. Then, by 
\cite[Remark, p.235]{DrR}, the quotient operad 
$\ip:=\uas/\I_E$ has generating series
$$\sum_{i=0}^{\infty} \frac{t^{2i}}{(1-t)^{2i+1}} 
(=1+\frac{t}{1-2t}).$$
Using the truncation ideals of $\ip$ we obtain codimension 
series of the form
\begin{equation}
\label{E1.1.1}\tag{E1.1.1}
\sum_{i=0}^{w} \frac{t^{2i}}{(1-t)^{2i+1}}
\end{equation}
which correspond to an increasing sequence of minimal 
varieties (with associated grade $2w$). More precisely, 
let $\J_{2w}$ be the ideal $\I_{E}+\up{2w+1}$
(which is also equal to $\I_{E}+\up{2w+2}$). Then 
$G_{\uas/J_{2w}}(t)$ is \eqref{E1.1.1}. By using 
\cite[Theorem 0.2]{BXYZZ}, we obtain algebras with the above 
codimension series. If we take $w=2$, then we obtain the 
codimension series in Theorem \ref{yythm0.4}(a).

Further, it is easy to check that $\gd(\J_{2w})=2w+2$
and $\mdeg(\J_{2w})=3$. So $\gd(\J_{2w})-\mdeg(\J_{2w})
=2w-1$ can be arbitrary large.
\end{rmk}

\subsection{Results from \cite{BYZ} which are needed in this paper}
\label{yysec1.6}
The truncation ideals play an important role in the study of 
growth property of 2-unitary operads. The following key 
lemma is \cite[Theorem 0.1(2)]{BYZ}. If we are working 
with GK-dimension or generating series of an operad $\ip$,
we implicitly assume that $\ip$ is locally finite
(meaning $\dim \ip(n)<\infty$ for all $n$).

\begin{lem}
\label{yylem1.2}
Let $\ip$ be a $2$-unitary operad and $\up{k}$ the $k$-th 
truncation ideal of $\ip$. Then
\[\gkdim \ip=\max\{k\mid \up{k}\neq 0\}+1 
=\min\{k\mid \up{k}=0\}.\]
\end{lem}

We collect some known results on $2$-unitary operads. For 
every subset $I=\{i_1, \cdots, i_s\}\subset [n]$ with 
$i_1<\cdots<i_s$, we denote a permutation
\[c_I\colon=(1, \cdots, i_1-1, i_1+1, 
\cdots, i_s-1, i_s+1, \cdots, n, i_1, \cdots, i_s)\in \S_n.\]

Let $\ip$ be a $2$-unitary operad and $\I$ an ideal of 
$\ip$. Let $n\ge k\ge 0$ be integers. Suppose that 
$\{\theta_i\mid 1\le i\le f_{\I}(k)\}$ be a $\Bbbk$-linear 
basis of $(\up{k}\cap \I)(k)$. Then, by \cite[Theorem 4.5]{BYZ},
\[\{[(\1_2\circ (\theta_i, \1_{n-k}))\ast c_I]
\mid 1\le i\le f_\I(k), I\subset [n], |I|=k\}\]
forms a $\Bbbk$-linear basis of 
$\dfrac{\up{k}\cap \I}{\up{k+1}\cap \I}(n)$. Consequently,
\[\dim \dfrac{\up{k}\cap \I}{\up{k+1}\cap \I}(n)
=f_\I(k){n\choose k}.\]

Clearly, when $\I=\ip$, we have 
\[\{[(\1_2\circ (\theta_i, \1_{n-k}))\ast c_I]
\mid 1\le i\le f_\ip(k), I\subset [n], |I|=k\}\]
is a $\Bbbk$-linear basis of $\up{k}(n)/\up{k+1}(n)$, where 
$\{\theta_i\mid 1\le i\le f_\ip(k)\}$ be a basis for 
$\up{k}(k)$. Hence we have the following consequences.

\begin{thm}
\cite[Theorem 4.6(1)]{BYZ}
\label{yythm1.3}
Let $\ip$ be a $2$-unitary operad. For each $k\ge 0$, let
\[\Theta^k\colon =\{\theta_1^k, \cdots, \theta_{z_k}^k\}\]
be a $\Bbbk$-linear basis for $\up{k}_\ip(k)$. Let 
$\textbf{B}_k(n)$ be the set
\[\{\1_2\circ (\theta_i^k, \1_{n-k})\ast c_I
\mid 1\le i\le z_k, I\subset [n], |I|=n-k\}\]
Then $\ip(n)$ has a $\Bbbk$-linear basis
\[\bigcup\limits_{0\le k\le n} \textbf{B}_k(n)
=\{\1_n\}\cup \bigcup\limits_{1\le k\le n} \textbf{B}_k(n),\]
and for every $k\ge 1$, $\up{k}_\ip(n)$ has a 
$\Bbbk$-linear basis $\bigcup_{k\le i\le n} \textbf{B}_i(n)$.
\end{thm}

\begin{lem}
\cite[Lemma 5.2]{BYZ}
\label{yylem1.4}
Let $\ip$ be a $2$-unitary operad and 
$\gamma_\ip(k):=\dim \up{k}_\ip(k)$ 
for each $k\ge 0$. Then the following hold.
\begin{enumerate}
\item[(1)] 
$G_\ip(t)=\sum\limits_{k=0}^\infty 
\gamma_\ip(k)\dfrac{t^k}{(1-t)^{k+1}}$.
\item[(2)] 
$\gkdim \ip=\max\{k\mid \gamma_\ip(k)\neq 0\}+1$.
\item[(3)] 
If $\gkdim\ip=d$, then 
$\dim\ip(n)=\sum\limits_{k=0}^{d-1}\gamma_\ip(k)
{n\choose k}$.
\end{enumerate}
\end{lem}

On a quotient operad of a 2-unitary operad, we have the following 
results.

\begin{lem}
\cite[Lemma 4.7]{BYZ}
\label{yylem1.5}
Let $\ip$ be a 2-unitary operad and $\I$ be an ideal of $\ip$. 
Then $\up{k}_{\ip/\I}\cong \up{k}_\ip/(\up{k}_\ip\cap \I)$.
\end{lem}

Combining Lemma \ref{yylem1.2} and Lemma \ref{yylem1.5}, 
we have the following fact.

\begin{cor}
\label{yycor1.6}
Let $\ip$ be a 2-unitary operad and $\I$ be a proper ideal of 
$\ip$. Then $\gkdim(\ip/\I)=k$ if and only if 
$\up{k-1}_\ip\not\subset \I$ and $\up{k}_\ip\subset \I$. 
\end{cor}

Next we focus on the truncation ideals of $\uas$. 

\begin{lem}
\cite[Corollary 5.4]{BYZ}
\label{yylem1.7}
Let $\I$ be an ideal of $\uas$. Then
$\gkdim(\uas/\I)\le k$ if and only if $\up{k}\subset \I$ for 
$k\ge 1$. Consequently, 
\[\gkdim(\uas/\up{k})=\begin{cases}
1, & k=1, 2, \\
k, & k\ge 3.
\end{cases}\]
\end{lem}

From  \cite[Section 5]{BYZ}, the binomial transform 
of the dimension sequence $\{\dim\uas(n)\}_{n\ge 1}$ is the   
sequence $\{\dim\up{n}(n)\}_{n\ge 1}$. Consequently, we have 
the following results.

For every $n\geq 1$, define
\begin{equation}
\label{E1.7.1}\tag{E1.7.1}
\gamma_n=\sum_{s=0}^{n} (-1)^{n-s} s! {n \choose s}.
\end{equation}

\begin{cor}
\label{yycor1.8}
Let $\up{k}$ be the $k$-th truncation ideal of $\uas$.	For 
$k\ge 1$, we have
\[\dim\up{k}(n)=
\begin{cases}
0 & n<k,\\
\gamma_k, & n=k, \\
\sum_{i=k}^n {n\choose i}\gamma_i, & n>k,	
\end{cases}\]
where $\gamma_i$ is given in \eqref{E1.7.1}.
\end{cor}

By the proof of \cite[Theorem 5.3 and Corollary 5.4]{BYZ}, 
we have

\begin{prop}\cite{Sp, BYZ}
\label{yypro1.9}
Let $\ip=\uas/\up{d}$ for $d\geq 3$. Then
\begin{equation}
\notag
G_{\ip}(t)=\sum_{n=0}^{d-1} 
\gamma_n \frac{t^n}{(1-t)^{n+1}}.
\end{equation}
\end{prop}

By above, $\uas/\up{d}$, for $d\geq 3$, provides a family of 
codimension series of type $(d-1,1)$. By 
\cite[Proposition 2.1]{GMP}, we have 
$$\Lambda_{(g,1)}\subset
\begin{cases}
\{ \frac{u}{g!} \mid 
1\leq u\leq \gamma_{g}, u\in \N\}, & {\text{$g$ is even}};\\
\{ \frac{u}{g!} \mid 
g-1\leq u\leq \gamma_{g}, u\in \N\}, & {\text{$g$ is odd}}.
\end{cases}
$$
When $g$ is odd, $\Lambda_{(g,1)}$ may not be the set 
$\{ \frac{u}{g!} \mid g-1\leq u\leq \gamma_{g}, u\in \N\}$, 
see comments at the end of \cite{GMP} and Theorem 
\ref{yythm5.15}. If $g\geq 4$ is even, $\Lambda_{(g,1)}$ may 
be a proper subset of $\{\frac{u}{g!} \mid 1\leq u\leq 
\gamma_{g}, u\in \N\}$, see Lemma \ref{yylem4.3}(2).

\subsection{More lemmas}
\label{yysec1.7}

%


\begin{lem}
\label{yylem1.10}
Let $\ip$ be a 2-unitary operad and $\I\subset \J$ 
be two ideals of $\ip$. 
\begin{enumerate}
\item[(1)]{\rm{[Basis Theorem for $\I$]}}
For each $k\ge 1$, let
\[\Theta^k_{\I}\colon =\{\theta_1^k, \cdots, \theta_{s_k}^k\}\]
be a $\Bbbk$-linear basis for $\I(k)\cap \up{k}_\ip(k)$. 
Let $\textbf{B}_{\I,k}(n)$ be the set
\[\{\1_2\circ (\theta_i^k, \1_{n-k})\ast c_I
\mid 1\le i\le s_k, I\subset [n], |I|=n-k\}.\]
Then $\I(n)$ has a $\Bbbk$-linear basis
\[\bigcup\limits_{1\le k\le n} \textbf{B}_{\I,k}(n).\]
\item[(2)]
Suppose that $\I(k)\cap \up{k}_\ip(k)$ is generated by $\zeta_k$ as a $\Bbbk\S_k$-module 
for $1\leq k\leq m$ {\rm{(}}we take $\zeta_k=0$ if $\I(k)\cap \up{k}_\ip(k)=0${\rm{)}}. 
Denote $\zeta^{[m]}\colon=\sum_{i=1}^{m} \1_2\circ(\zeta_i, \1_{m-i})$. 
Then $\lr{\zeta^{[m]}}(k)=\I(k)$ for all $k\leq m$.
As a consequence, if further, $\I$ is generated by $\I(m)$, 
then $\I=\lr{\zeta^{[m]}}$.
\item[(3)]
If $\up{m}_\ip=0$, then $\I$ is generated by $\I(m)$ {\rm{(}}and by $\I(m+1)${\rm{)}}. 
\item[(4)]
Suppose that $\I$ is generated
by $\I(m)$ and that $\J/\I$ is generated by $(\J/\I)(m)$ for
some $m\geq 1$. Then $\J$ is generated by $\J(m)$. 
\end{enumerate}
\end{lem}

\begin{proof}
(1) Let $\overline{\ip}$ denote the quotient operad $\ip/\I$.
Extend the basis $\Theta^k_{\I}$ of $\I\cap \up{k}(k)$ to 
a basis of $\up{k}(k)$, say
$\{\theta_1^k, \cdots, \theta_{s_k}^k, \theta_{s_k+1}^k,
\cdots, \theta_{z_k}^k\}$.
Then images of $\{\theta_{s_k+1}^k,\cdots, \theta_{z_k}^k\}$ serve as a 
basis of $\up{k}_{\overline{\ip}}(k)$ by Lemma \ref{yylem1.5}. 
Let $\textbf{B}_{\I',k}(n)$ be the set
\[\{\1_2\circ (\theta_i^k, \1_{n-k})\ast c_I
\mid s_k+1\le i\le z_k, I\subset [n], |I|=n-k\}\]
for $n\geq 0$. By Theorem \ref{yythm1.3}, 
$\bigcup\limits_{0\le k\le n} \textbf{B}_{\I',k}(n) 
\cup \textbf{B}_{\I,k}(n)$ is a basis of $\ip(n)$. 
On the other hand, then images of 
$\bigcup\limits_{0\le k\le n} \textbf{B}_{\I',k}(n)$ serve
as a basis of $\overline{\ip}(n)$. Therefore 
$\bigcup\limits_{0\le k\le n} \textbf{B}_{\I,k}(n)$
(which is equal to $\bigcup\limits_{1\le k\le n} \textbf{B}_{\I,k}(n)$)
is a basis of $\I(n)$. 

(2) We use induction on $m$. If $m=1$, then $\lr{\zeta_1}(1)=
\I\cap \up{1}(1)$ which has a basis $\Theta^1_{\I}$. By part (1),
$B_{\I,1}(1)=\Theta^1_{\I}$ is also a basis of $\I(1)$. So
the assertion holds. It is clear that 
$\zeta^{[m]}=\zeta^{[m-1]}\ucr{m-1} \1_2+\zeta_m$. Then 
$\zeta^{[m]}\ucr{m} \1_0=(\zeta^{[m-1]}\ucr{m-1} \1_2)\ucr{m} \1_0
+\zeta_m\ucr{m}\1_0=\zeta^{[m-1]}$. This implies that both
$\zeta^{[m-1]}$ and $\zeta_{m}$ are in $\lr{\zeta^{[m]}}$.
So by induction 
hypothesis, for all $k\leq m-1$, $\I(k)=\lr{\zeta^{[m-1]}}(k)
\subset \lr{\zeta^{[m]}}(k)\subset \I(k)$. By part (1),
$B_{\I,k}(m) \subset \lr{\zeta^{[m]}}(m)$ for all
$1\leq k\leq m-1$. Since $B_{\I,m}(m)=\Theta^{m}_{\I}$ is 
generated by $\zeta_{m}$, the assertion follows from part (1).

The consequence is clear.

(3) This follows from part (1).

(4) This is clear.
\end{proof}

If $\ip$ is $\uas$ (or a quotient operad of $\uas$), then
every $(\I\cap \up{k})(k)$ is generated by one element 
as it is a submodule of $\Bbbk\S_{k}$. 

\begin{Def}
\label{yydef1.11}
Let $\I$ be an ideal of an operad. The 
{\it minimal degree} of $\I$ is defined to be
$$\mdeg(\I)=\min\{ n\mid \I(n)\neq 0\}.$$
\end{Def}

\begin{lem}
\label{yylem1.12}
Let $\ip$ be a 2-unitary operad and $\I$ be an ideal of $\ip$.
Then $$\mdeg(\I)=\min\{n \mid \I(n)\cap \up{n}(n)\neq 0\}
=\max\{n\mid \I\subset \up{n}\}.$$
\end{lem}

\begin{proof}
Suppose $\I\neq 0$ and $\mdeg(\I)=m<\infty$. Then 
$\I(m-1)=0$. By the definition of the restriction operator 
$\pi^{I}$, one sees that $\pi^{I}(\I)=0$ when $|I|=m-1$.
Hence $\I\subset \up{m}$, in particular, $\I(m)\cap 
\up{m}(m)=\I(m)\neq 0$. Hence the assertions follow.
\end{proof}

\section{Connections between PI-algebras and quotient 
operads of $\uas$}
\label{yysec2}

In this section, we recall some connections between 
PI-algebras (or T-ideals) and (operadic) ideals of $\uas$.
A part of this section is copied from \cite{BXYZZ}.

Observe that a PI-algebra $A$ always satisfies a multilinear 
polynomial identity of degree $\le d$ if $A$ satisfies an 
identity of degree $d$, see \cite[Proposition 13.1.9]{MR} 
or \cite[Theorem 1.3.7]{GZ6}. A \textit{multilinear polynomial} 
of degree $n$ is a nonzero element $f(x_1, \cdots, x_n)\in 
\Bbbk\langle x_1, \cdots, x_n \rangle$ of the form
\[f(x_1, \cdots, x_n)=\sum_{\sigma\in \S_n} 
c_\sigma x_{\sigma(1)} \cdots x_{\sigma(n)}\]
for some $c_{\sigma}\in \Bbbk$. The method of 
multilinearization actually plays a very important role in 
the study of the identities of a PI-algebra.

\begin{lem}
\cite[Lemma 2.1]{BXYZZ}
\label{yylem2.1}
Let $A$ be an associative algebra. 
Then $A$ is a PI-algebra if and only if $A$ is an $\uas/\I$-algebra 
for some nonzero operadic ideal $\I$ of $\uas$.
\end{lem}

Let $\Bbbk\langle X \rangle$ be the free algebra generated 
by the set $X=\{x_i\}_{i\in \N_+}$ and $S$ a non-empty 
subset of $\Bbbk\lr{X}$. Recall a \textit{variety 
$\mathcal{V}=\mathcal{V}(S)$ determined by $S$} is the 
class of all algebras $A$ such that $f(a_1, \cdots, a_n)=0$
$(\forall a_1, \cdots, a_n\in A )$ for all 
$f(x_1, \cdots, x_n)\in S$. If $\mathcal{V}$ is a variety 
and $A$ is an $\Bbbk$-algebra such that 
$\Id(A)=\Id(\mathcal{V})$, then we say that $\mathcal{V}$ 
is the variety generated by $A$ and write 
$\mathcal{V}=\operatorname{var}(A)$. Clearly, 
$\mathcal{V}=\operatorname{var}(A)$ is exactly the class 
of $\uas/\I_A$-algebras, where $\I_A$ is the kernel of the 
operadic homomorphism $\gamma_A \colon \uas\to \iend_A$.

Recall that an ideal $H$ of $\Bbbk\langle X \rangle$ is 
called a \textit{T-ideal} if 
$\varphi(H)\subset H$ for every endomorphism $\varphi$ of 
$\Bbbk\langle X \rangle$. Let $A$ be a PI-algebra. The set 
$\Id(A)$ of all polynomial identities of $A$ in 
$\Bbbk\langle X \rangle$ is a T-ideal. Here is the famous
result of Kemer on finite generation of T-ideals.

\begin{thm}
\cite[Theorem 2.4]{Ke}
\label{yythm2.2}
Every T-ideal over a field of characteristic zero is finitely
generated as a T-ideal.
\end{thm}
 
When we consider PI-algebras with unit, the authors proved
that every T-ideal over a field of characteristic zero is 
generated by a single element as a T-ideal 
\cite[Theorem 0.1]{BXYZZ}. The authors also proved the following result.

\begin{thm}
\cite[Theorem 2.9]{BXYZZ} 
\label{yythm2.3}
There is a natural inclusion-preserving one-to-one 
correspondence between the set of proper T-ideals of 
$\Bbbk\lr{X}$ and the set of proper ideals of $\uas$.
\end{thm}

Recall that a variety $\mathcal{V}$ of algebras is said to be 
{\it minimal} of polynomial growth $n^k$ if any subvariety of 
$\mathcal{V}$ of polynomial growth $n^t$ with $t<k$. 
Equivalently, we introduce the corresponding definition on 
operadic ideals.
	
\begin{Def}
\label{yydef2.4}
Let $\I$ be an ideal of an operad $\ip$ and $\gkdim(\ip/\I)=d$.
If for any ideal $\mathcal{J}$ of $\ip$ strictly containing 
$\I$,  $\gkdim(\ip/\mathcal{J})<d$, then we say that $\I$ 
\textit{is maximal with respect to}  $\gkdim$.
\end{Def}
	
\begin{prop}
\label{yypro2.5}
Let $A$ be a PI-algebra with associated to the operadic 
ideal $\I_A$. Then the variety $\mathcal{V}$ generated 
by $A$ is minimal of polynomial growth $n^k$ if and 
only if $\I_A$ is maximal with respect to $\gkdim$
{\rm{(}}with $\gkdim(\ip/\I_A)=k+1${\rm{)}}.
\end{prop}

For the rest of this section we prove Proposition 
\ref{xxpro0.5}.

\begin{proof}[Proof of Proposition \ref{xxpro0.5}]
Let ${\mathbb F}$ be the set of all finitely generated 
fields over ${\mathbb Q}$ up to isomorphism. Then 
the set ${\mathbb F}$ is countable and every member 
$F\in {\mathbb F}$ is a countable field. For each 
$F\in {\mathbb F}$, let $OI(F)$ be the set of operadic 
ideals of $\uas_F$ over the field $F$. Since $\uas_F$ 
is countable and every operadic ideal is generated 
by one element \cite[Theorem 0.3(2)]{BXYZZ}, $OI(F)$ 
is countable. Let ${\mathbb B}$ be the set of all 
generating series of quotient operads of $\uas_F$ 
for all $F\in {\mathbb F}$. 
Then ${\mathbb B}$ is countable. 

Now let $A$ be any PI-algebra over a base field 
$\Bbbk$ (where $\Bbbk$ is not necessarily countable).
It remains to show that $c(A,t)$ is in ${\mathbb B}$. 
Let $\I_{A}$ be the operadic ideal of $\uas(=\uas_{\Bbbk})$
associated to $A$. Then we have
\begin{equation}
\notag 
c(A,t)=G_{\uas/\I_A}(t).
\end{equation}
Hence it suffices to show that $G_{\uas/\I_A}(t)
\in {\mathbb B}$.

Let $\Id_{\Bbbk}(A)$ be all polynomial identities of $A$ 
in $\Bbbk\lr{X}$ which is a T-ideal. By Kemer's 
representability Theorem \cite{Ke} (also see 
\cite[Theorem 1.1]{AKK}), there exists a field extension 
$L$ of $\Bbbk$ and a finite dimensional algebra $R$ over 
$L$ which is PI equivalent to $A$, that is $Id_{L}(A_L) 
=Id_{L}(R)$, where $A_L=A\otimes_{\Bbbk}L$. Note that
$\I_{A_L}\cong \I_{A}\otimes_{\Bbbk} L$. Since 
we only consider the generating series of
$\uas/\I_A$, we might assume $\Bbbk=L$ without loss of
generality. Then it follows from Kemer's 
representability Theorem that we can assume $A$ is finite
dimensional over $\Bbbk$. In this case $A=S \otimes_F \Bbbk$
where $S$ is a finite dimensional algebra over a field 
$F$ and $\Bbbk\supset F\in {\mathbb F}$. Then 
$\I_{A}=\I_{S}\otimes_{F} \Bbbk$. Thus $G_{\uas/\I_A}(t)
=G_{\uas_{F}/\I_S}\in {\mathbb B}$ as required.
\end{proof}

\section{The generating degrees of truncation ideals}
\label{yysec3}

In general it is extremely difficult to determine a 
generating identity and the generating degree of an 
operadic ideal of $\uas$. The main goal of this 
section is to understand the generating degree of 
the truncation ideals of $\uas$. Throughout this 
section $\up{k}$ denotes the $k$-th truncation ideal of $\uas$. 

\subsection{Specht bases of truncation ideals}
\label{yysec3.1}
Let $V_n$ be the subspace of $\Bbbk\langle x_1, \cdots, 
x_n\rangle$ consisting of all multilinear polynomials of 
degree $n$. Recall that a polynomial $f\in \Bbbk \langle 
x_1, \cdots, x_n\rangle$ is called a 
\textit{proper polynomial} if it is a linear combination 
of products of (left normed) commutators
\[f(x_1, \cdots, x_n)
=\sum \lambda_{i_{1{\ast}}, \cdots, i_{m{\ast}}}
[x_{i_{11}}, \cdots, x_{i_{1k_1}}]\cdots
[x_{i_{m1}}, \cdots, x_{i_{mk_m}}], \quad 
\lambda_{i_{1{\ast}}, \cdots, i_{m{\ast}}}\in \Bbbk,\]
where the commutator of length $k\ge 2$ is defined inductively 
by 
\begin{align*}
[x_{i_1}, x_{i_2}]\colon= & x_{i_1}x_{i_2}-x_{i_2}x_{i_1}, \\
[x_{i_1}, \cdots, x_{i_{k-1}}, x_{i_k}]\colon=
& [[x_{i_1}, \cdots, x_{i_{k-1}}], x_{i_k}], k>2.
\end{align*}
We denote by $B_n$ the set of all proper polynomials in 
$\Bbbk \langle x_1, \cdots, x_n\rangle$, and
\[\Gamma_n=B_n\cap V_n.\]


\begin{lem}
\cite[Theorem 4.3.9]{Dr}
\label{yylem3.1}
A basis of the vector space $\Gamma_n$ of all proper 
multilinear polynomials of degree $n\ge 2$ consists of 
the following product of commutators
\[[x_{i_{11}}, \cdots, x_{i_{1k_1}}]\cdots
[x_{i_{m1}}, \cdots, x_{i_{mk_m}}],\]
satisfying
\begin{enumerate}
\item[(1)] 
$\{i_{11}, \cdots, i_{1k_1}, \cdots, i_{m1}, \cdots, 
i_{mk_m}\}=[n]$; 
\item[(2)] 
$i_{j1}>i_{j2}, \cdots, i_{jk_j}$ for all $j=1, \cdots, m$;
\item[(3)] 
$2\le k_1\le \cdots \le k_m$;
\item[(4)] 
if $k_t=k_{t+1}$ for some $t\in[m-1]$, then 
$i_{t1}<i_{t+1, 1}$.
\end{enumerate}
\end{lem}


The above basis of $\Gamma_n$ is called the 
\textit{Specht basis} \cite[p.45]{Dr}. Observe that
$V_n$ admits the right $\Bbbk\S_n$-action given by
\[(x_{i_1}\cdots x_{i_n})\ast \rho\colon
=x_{\rho^{-1}(i_1)} \cdots x_{\rho^{-1}(i_n)}\]
for $\rho\in \S_n$, and

\begin{align}
\label{E3.1.1}\tag{E3.1.1}
\Phi_n\colon \uas(n)  \to V_n , \quad \quad 
\sigma \mapsto x_{\sigma^{-1}(1)}\cdots 
x_{\sigma^{-1}(n)}
\end{align}
is an isomorphism of right $\Bbbk\S_n$-modules. 
Denote by $\Psi_n$ the inverse of $\Phi_n$. Clearly, 
$\Gamma_n$ is a $\Bbbk\S_n$-submodule of $V_n$, and 
\[\Phi_n(\tau_n)=[x_1, x_2 \cdots, x_n]\]
where $\tau_{n}\in \Lie(n)\subset \uas(n)$ is 
introduced in \eqref{E1.0.2}.

\begin{lem}
\label{yylem3.2}
Retain the above notation. Then the restriction 
of $\Psi_n$ on $\Gamma_n$ induces an isomorphism 
\[\Psi_n|_{\Gamma_n}\colon  \Gamma_n \to \up{n}(n).\]
\end{lem}

\begin{proof}
For any $[x_{i_{11}}, \cdots, x_{i_{1k_1}}]\cdots
[x_{i_{m1}}, \cdots, x_{i_{mk_m}}]\in \Gamma_n$, we have 
\[\Psi_n([x_{i_{11}}, \cdots, x_{i_{1k_1}}]\cdots[x_{i_{m1}}, 
\cdots, x_{i_{mk_m}}])=(1_s\circ(\tau_{k_1}, \cdots, 
\tau_{k_m}))\ast 
(i_{11}, \cdots, i_{1k_1},\cdots, i_{m1}, \cdots, i_{mk_m}),\]
where $\tau_{k_i}\in \Lie(k_i)\subset \uas(k_i)$ is 
introduced in Section 1.2.
Therefore, by $\tau\ucr{1}1_0=\tau\ucr{2} 1_0=0$, we get 
\[\Psi_n([x_{i_{11}}, \cdots, x_{i_{1k_1}}]
\cdots[x_{i_{m1}}, \cdots, x_{i_{mk_m}}])\ucr{j} 1_0=0.\]
for any $1\le j\le n$, and $\Psi_n(\Gamma_n)\subset \up{n}(n)$.  
By \cite[Theorem 14]{Sp} and Corollary \ref{yycor1.8}, we know
\[\dim \Gamma_n=\sum\limits_{i=0}^n(-1)^i\dfrac{n!}{i!}
=\dim\up{n}(n) =\gamma_n.\]
Since $\Psi_n$ is an isomorphism, the restriction 
$\Psi_n|_{\Gamma_n}\colon \Gamma_n \to \up{n}(n)$ is injective 
and therefore it is also isomorphism.
\end{proof}

By the isomorphism $\Psi_n|_{\Gamma_n}$ (also denoted by 
$\Psi_n$ for brevity) and the Specht basis of $\Gamma_n$, 
we can give a basis of $\up{n}(n)$. Suppose that 
$n, k_1, \cdots, k_m$ are integers satisfying $k_1+\cdots+k_m=n$.
Denote by $\Sigma_{k_1, \cdots, k_m}$ the subset  of $\S_n$ 
consisting of the permutation 
\[(i_{11}, \cdots, i_{1k_1},\cdots, i_{m1}, \cdots, i_{mk_m})\]
satisfying the conditions (2)-(4) given in Lemma \ref{yylem3.1}.
For brevity, we denote 
\begin{align}\label{E3.2.1}\tag{E3.2.1}
1_m\circ(\tau_{k_1},  \cdots, \tau_{k_m})\colon 
=\tau_{k_1, \cdots, k_m},
\end{align}
and 
\[(1_m\circ(\tau_{k_1},  \cdots, \tau_{k_m}))\ast \sigma\colon 
=\tau_{k_1, \cdots,  k_m}^\sigma\]
for $\sigma\in\S_n$.
Clearly, 
\[\Psi_n([x_{i_{11}}, \cdots, x_{i_{1k_1}}]\cdots
[x_{i_{m1}}, \cdots, x_{i_{mk_m}}])
=\tau_{k_1, \cdots, k_m}^\sigma\]
for $\sigma=(i_{11}, \cdots, i_{1k_1},\cdots, i_{m1}, 
\cdots, i_{mk_m})\in \S_n$.

\begin{prop}
\label{yypro3.3}
Let $\up{n}$ be the $n$-th truncation ideal of $\uas$. 
Then the set 
\[\Theta^n\colon=\{\tau_{k_1, \cdots, k_m}^\sigma\mid 
\sigma\in\Sigma_{k_1, \cdots, k_m}, k_1+\cdots+k_m=n\}\]
is a basis of $\up{n}(n)$, which is also called the 
\textit{Specht basis} of $\up{n}(n)$.
\end{prop}

Therefore, each $\theta\in \up{n}(n)$ can be linearly 
represented by $\Theta^n$. We write
\[\sl(\theta)\colon=\min\{k_m\mid \theta
=\sum_{\tau_{k_1, \cdots, k_m}^\sigma\in \Theta^n}
\la^\sigma_{k_1, \cdots, k_m}\tau_{k_1, \cdots, k_m}^\sigma,  
{\text{ where }} 0\neq \la^\sigma_{k_1, \cdots, k_m}\in \Bbbk\},\]
which is called the \textit{Specht length} of 
$\theta\in \up{n}(n)$, and define 
\[M^n_t\colon =\{\theta\in \up{n}(n)\mid \sl(\theta)\ge t\}.\] 
for $2\le t\le n$.

Let $I$ be a subset of $[n]$ consisting of $f$ elements. 
Denote by
\[\overrightarrow{I}\colon [f]\to I\subset [n], \]
the order-preserving map, namely, $\overrightarrow{I}(k)=i_k$ 
for $1\le k\le f$ if $I=\{i_1,\cdots, i_f\}$ and $i_1<\cdots<i_f$.

\begin{lem}
	\label{yylem3.4}
	Retain the above notation.  If $l_1+\cdots+l_m=n$, $l_j\ge 2$, 
	$j=1, \cdots, m$, and there exists $l_j\ge t$, then 
	\[(1_m\circ (\tau_{l_1}, \cdots, \tau_{l_m}))\ast\rho \in M_t^n\]
	for any $\rho\in \S_n$. 
	Consequently, $M_t^n$ is a $\Bbbk\S_n$-submodule of $\up{n}(n)$.
\end{lem}

\begin{proof}
	Applying the isomorphism $\Psi_n \colon \Gamma_n \to \up{n}(n)$, 
	we need only to show that 
	\[\Psi_n([x_{j_{11}}, \cdots, x_{j_{1l_1}}]\cdots 
	[x_{j_{m1}}, \cdots, x_{j_{ml_m}}])\in M_t^n\]
	for any $(j_{11}, \cdots, j_{1l_1}, \cdots, j_{m1}, 
	\cdots, j_{ml_m})\in \S_n$ with some $l_p\ge t$. For convenience, 
	we denote $J_s\colon =\{j_{s1}, \cdots, j_{s l_s}\}$ for 
	$1\le s\le m$.
	
	If $j_{s1}<j_{sr}$ for some $2\le r\le l_s$, we consider the 
	order-preserving map $\overrightarrow{J_s}\colon [l_s]\to J_s$. 
	Since $\{\Psi_{l_s}([x_{l_s}, x_{i_1}, \cdots, x_{i_{l_s-1}}])
	\mid (i_1, \cdots, i_{l_s-1})\in \S_{l_{s}-1}\}$ is a 
	$\Bbbk$-linear basis of $\Lie(l_s)$, there exist scalars 
	$\la_{i_1, \cdots, i_{l_s-1}}$'s such that 
	\begin{align*}
		[x_{\overrightarrow{J_s}^{-1}(j_{s1})}, \cdots, 
		x_{\overrightarrow{J_s}^{-1}(j_{sl_{s}})}]
		=\sum_{(i_1, \cdots, i_{l_s-1})\in \S_{l_{s}-1}}
		\la_{i_1, \cdots, i_{l_s-1}} [x_{l_s}, x_{i_1}, 
		\cdots, x_{i_{l_s-1}}],
	\end{align*}
	and hence 
	\begin{align*}
		[x_{j_{s1}}, \cdots, x_{j_{sl_{s}}}]
		=\sum\la_{i_1, \cdots, i_{l_s-1}} 
		[x_{\overrightarrow{J_s}(l_s)}, x_{\overrightarrow{J_s}(i_1)}, 
		\cdots, x_{\overrightarrow{J_s}(i_{l_s-1})}].
	\end{align*}
	Therefore, we can assume that $j_{s1}>j_{s2}, \cdots, j_{sl_s}$ 
	for all $s=1, \cdots, m$. 
	
	If $l_{s-1}>l_s$, or 
	$l_{s-1}=l_s$ and $j_{s-1\; l_{s-1}}>j_{s\; l_s}$ for some $s$, 
	we have
	\begin{align*}
		[x_{j_{s-1\;1}}, \cdots, &x_{j_{s-1\; l_{s-1}}}]
		[x_{j_{s1}}, \cdots, x_{j_{s\; l_{s}}}]\\
		=& [x_{j_{s1}}, \cdots, x_{j_{sl_{s}}}]
		[x_{j_{s-1\; 1}}, \cdots, x_{j_{s-1\; l_{s-1}}}]+[[x_{j_{s-1\; 1}}, 
		\cdots, x_{j_{s-1\; l_{s-1}}}], 
		[x_{j_{s1}}, \cdots, x_{j_{s\; l_{s}}}] ].
	\end{align*}
	Considering the order-preserving map 
	$\overrightarrow{J_{s-1}\cup J_{s}}\colon 
	[l_{s-1}+l_s]\to J_{s-1}\cup J_{s}$ (which is denoted by 
	$\overrightarrow{J_{s-1,s}}$ in the next expression),
	we have that 
	\[\Psi_{l_{s-1}+l_s}([[x_{\overrightarrow{J_{s-1,s}}^{-1}(j_{s-1\;1})}, 
	\cdots, x_{\overrightarrow{J_{s-1,s}}^{-1}(j_{s\;l_{s-1}})}], 
	[x_{\overrightarrow{J_{s-1,s}}^{-1}(j_{s\;1})}, \cdots, 
	x_{\overrightarrow{J_{s-1,s}}^{-1}(j_{s\;l_{s}})}]])\in 
	\Lie(l_{s-1}+l_s),\]
	and there exist scalars $\la_{i_{s-1\;1}, \cdots, i_{s-1\; l_{s-1}}, 
		i_{s1}, \cdots, i_{s\; l_s}}$'s such that 
	\begin{align*}
		[[x_{j_{s-1\; 1}}, \cdots, & x_{j_{s-1\; l_{s-1}}}], 
		[x_{j_{s1}}, \cdots, x_{j_{s\; l_{s}}}] ]\\
		=& \sum \la_{i_{s-1\;1}, \cdots, i_{s-1\; l_{s-1}}, i_{s1}, 
			\cdots, i_{s\; l_s}} 
		[x_{\overrightarrow{J_{s-1,s}}(l_{s-1}+l_s)},
		x_{\overrightarrow{J_{s-1,s}}(i_1)}, \cdots, 
		x_{\overrightarrow{J_{s-1,s}}(i_{l_{s-1}+l_s-1})}]
	\end{align*}
	Therefore, if
	\begin{align*}
		\Psi_n([x_{j_{11}}, \cdots, x_{j_{1l_1}}]\cdots 
		[x_{j_{m1}}, \cdots, x_{j_{ml_m}}])	
		=\sum_{k_1+\cdots+k_s=n,\atop \sigma\in \Sigma_{k_1, \cdots, k_n}}
		\la^\sigma_{k_1, \cdots, k_s} \tau^\sigma_{k_1, \cdots, k_s}
	\end{align*} 
	and $l_p\ge t$ for some $p$, then we have $k_s\ge t$ for 
	each $\la^\sigma_{k_1, \cdots, k_s}\neq 0$. It follows that 
	\[\Psi_n([x_{j_{11}}, \cdots, x_{j_{1l_1}}]\cdots 
	[x_{j_{m1}}, \cdots, x_{j_{ml_m}}])\in M_t^n.\] 
\end{proof}

Since each $M_t^n$ is a $\Bbbk\S_n$-submodule of $\up{n}(n)$ 
and $M_t^n\subset M_{t-1}^n$, we have the following result. 

\begin{prop}
	\label{yypro3.5}
	Let $\up{n}$ be the $n$-th truncation ideal of $\uas$. Then 
	$\up{n}(n)$ has a chain of submodules of $\Bbbk\S_n$:
	\[\Lie(n)=M^n_n\subset M^n_{n-1}\subset \cdots 
	\subset M^n_{2}=\up{n}(n).\]
\end{prop}

\subsection{Generators of the truncation ideals} 
\label{yysec3.2}
Concerning a general ideal of $\uas$ 
generated by a submodule of $\uas(n)$ for some $n$, we 
have the following fact. 


\begin{lem}
\label{yylem3.6}
Let $\ip$ be \; $\uas$ or a quotient operad of \; $\uas$ and 
$\I$ be an ideal of $\ip$. If $\I=\lr{\I(n)}$ 
for some $n$, then $\I=\lr{\I(m)}$ for any $m\ge n$.
\end{lem}

\begin{proof} 
Assume that $\I=\langle \I(m-1)\rangle$ for some $m>n$. Then, 
for any $\theta\in \I(m-1)$, $\theta\ucr{1} \1_2\in \I(m)$.
Therefore, $\theta=(\theta\ucr{1} \1_2)\ucr{1} \1_0\in \langle 
\I(m)\rangle(m-1)$. Consequently, $\I=\langle \I(m)\rangle$.
\end{proof}

Let $\I$ be an ideal of $\uas$. By Lemma \ref{yylem3.6} and
\cite[Theorem 0.3(2)]{BXYZZ}, we 
know there is a unique minimal degree $n$ such that 
$\I=\lr{\I(n)}$, which is called the \textit{generating degree} 
of $\I$, denoted by $\gd(\I)$ [Definition \ref{yydef0.7}]. 
Next, we will study the generating degrees of truncation ideals 
of $\uas$. 

Recall from Definition 3.2 in \cite{GMZ} that a polynomial 
$g$ is a consequence of a polymomial $f$ if $g$ lies in 
the T-ideal generated by $f$. In \cite{Ma3}, the author gave 
the generators of the space of the $n$-th multilinear proper 
polynomials of $G$-graded algebras for a finite group $G$. 
Taking $G$ as the trivial group, it immediately follows from 
\cite[Lemma 2.2]{Ma3} that for every $i\ge 1$, $\Gamma_{k+i}$ 
is a consequence of $\Gamma_{k}$ in case $k$ is even
and is a consequence of $\Gamma_{k}$ plus the polynomial 
$[x_1, x_2]\cdots[x_k, x_{k+1}]$ in case $k$ is odd.
Apply the isomorphism between $\Gamma_k$ and $\up{k}(k)$, 
and we have the following fact on the truncation ideals.


\begin{lem}
\label{yylem3.7}
Let $\up{k}$ be the $k$-th truncation ideal of $\uas$. 
\begin{enumerate}
\item[(1)] 
If $k\geq 2$ is even, then $\up{k+i}(k+i)\subset 
\langle \up{k}(k)\rangle (k+i)$ for every $i\geq 1$.
\item[(2)] 
If $k\geq 2$ is odd, then $\up{k+i}(k+i)\subset 
\langle \up{k}(k), \up{k+1}(k+1) \rangle (k+i)
=\lr{\up{k}(k+1)}(k+i)$ for every $i\geq 2$. 	
\end{enumerate}
\end{lem}

In this subsection, we study the generating degree 
of the truncation ideals of $\uas$. From Lemma 
\ref{yylem3.7}, it is easily seen that $\up{k}$ 
can be generated by $\up{k}(k)$ when $k$ is even,
and by $\up{k}(k+1)$ when $k$ is odd. Next, we 
will show that $\up{k}$ can not be generated by 
$\up{k}(k)$ in case $k$ is odd. For this, we first 
prove the following facts.

%

\begin{lem}
\label{yylem3.8}
Let $k\geq 3$ be an odd number. Then 
\begin{align}
\label{E3.8.1}\tag{E3.8.1}
\tau_{2, \cdots, 2}\ast (\sgn(\sigma)\sigma-1_{k+1})
\in \lr{\up{k}(k)}(k+1)
\end{align}
for any $\sigma\in \S_{k+1}$. 
\end{lem}

\begin{proof}
Applying the isomorphism 
$\Psi_{k+1}\colon V_{k+1}\to \Bbbk\S_{k+1}$, we need only 
to show
\[\Psi_{k+1}(\sgn(\sigma)[x_{\sigma^{-1}(1)}, x_{\sigma^{-1}(2)}]
\cdots [x_{\sigma^{-1}(k)}, x_{\sigma^{-1}(k+1)}]
-[x_1, x_2]\cdots[x_{k}, x_{k+1}])\in \lr{\up{k}(k)}(k+1).\]
We first prove the special case when $k=3$. By a direct calculation, 
we have
\begin{align*}
[x_1, x_3]&[x_2, x_4]+[x_1, x_2][x_3, x_4]\\
=& [[x_1, x_3]x_2, x_4]-[x_1, x_3, x_4]x_2
 +[[x_1, x_2]x_3, x_4]-[x_1, x_2, x_4]x_3\\
=& [[x_1x_2, x_3]-x_1[x_2, x_3], x_4]
 -[x_1, x_3, x_4]x_2+[[x_1x_3, x_2]-x_1[x_3, x_2], x_4]
 -[x_1, x_2, x_4]x_3\\
=& [x_1x_2, x_3, x_4]-[x_1, x_3, x_4]x_2
 +[x_1x_3, x_2, x_4]-[x_1, x_2, x_4]x_3.
\end{align*}
It follows that 
\[\tau_{2, 2}\ast (-(23)-1_{4})
=(1_2\ucr{1}\tau_3)\ast((243)+(34))
-(\tau_3\ucr{1}1_2)\ast(1_{4}+(23))\in \lr{\up{3}(3)}(4),\]
where $(23), (243), (34)$ are cycle permutations in 
$\S_{4}$. 

For the general case, we use induction on the inversion 
number $\phi(\sigma)$. By definition, when $\sigma$ is 
written as $(i_1, i_2, \cdots, i_{k+1})$ with 
$i_s=\sigma^{-1}(s)$ for all $s$, $\phi(\sigma)$ is the 
number of pairs $(i_{s},i_{s'})$ where $s<s'$ and $i_s>i_{s'}$. 
(Sometimes this is called the inversion number of 
$\sigma^{-1}$ by some authors).

If $\phi(\sigma)=0$, then \eqref{E3.8.1} holds obviously 
since $\sgn(\sigma)\sigma-(1)=0$.

Assume that \eqref{E3.8.1} holds when $\phi(\sigma)\le l$. 
When $\phi(\sigma)=l+1$, there are two cases. 

Case 1. There exists $s$ with $2s\le k+1$ such that 
$\sigma^{-1}(2s-1)>\sigma^{-1}(2s)$. 
Set \[\sigma'=(\sigma^{-1}(1), \cdots, \sigma^{-1}(2s-2), 
\sigma^{-1}(2s), \sigma^{-1}(2s-1),
\sigma^{-1}(2s+1), \cdots, \sigma^{-1}(k+1)).\] 
Then we have 
$\phi(\sigma')=\phi(\sigma)-1$, and 
\begin{align*}
 \tau_{2, \cdots, 2}\ast (\sgn(\sigma)\sigma-1_{k+1})
=\tau_{2, \cdots, 2}\ast (\sgn(\sigma')\sigma'-1_{k+1})
\in \lr{\up{k}(k)}(k+1)
\end{align*}
by the induction hypothesis.

Case 2. Suppose $\sigma^{-1}(2t-1)<\sigma^{-1}(2t)$
for all $2t\leq k+1$. Then there exists $s$ with 
$2s\le k$ such that $\sigma^{-1}(2s)>\sigma^{-1}(2s+1)$.
Write $\sigma=(i_1, \cdots, i_{k+1})$ and 
$\sigma'=(i_1, \cdots, i_{2s-1}, i_{2s+1}, 
i_{2s}, i_{2s+2}, \cdots, i_{k+1})$. 
Clearly, $\phi(\sigma')=\phi(\sigma)-1$.  
By the proof of the special case $k=3$, we have 
\begin{align*}
\sgn(\sigma)&[x_{i_1}, x_{i_2}]\cdots 
[x_{i_k}, x_{i_{k+1}}]-[x_1, x_2]\cdots[x_{k}, x_{k+1}]\\
=&\sgn(\sigma)[x_{i_1}, x_{i_2}]\cdots 
[x_{i_k}, x_{i_{k+1}}]-\sgn(\sigma')[x_{i_1}, x_{i_2}]
\cdots[x_{i_{2s-1}}, x_{i_{2s+1}}][x_{i_{2s}}, x_{i_{2s+2}}]
\cdots [x_{i_k}, x_{i_{k+1}}]\\
&+\sgn(\sigma')[x_{i_1}, x_{i_2}]
\cdots[x_{i_{2s-1}}, x_{i_{2s+1}}][x_{i_{2s}}, x_{i_{2s+2}}]
\cdots [x_{i_k}, x_{i_{k+1}}]-
[x_1, x_2]\cdots[x_{k}, x_{k+1}]\\
=&\sgn(\sigma)[x_{i_1}, x_{i_2}]\cdots 
([x_{i_{2s-1}}, x_{i_{2s}}][x_{i_{2s+1}}, x_{i_{2s+2}}]
+[x_{i_{2s-1}}, x_{i_{2s+1}}][x_{i_{2s}}, x_{i_{2s+2}}])
\cdots [x_{i_k}, x_{i_{k+1}}]\\
&+\sgn(\sigma')[x_{i_1}, x_{i_2}]
\cdots[x_{i_{2s-1}}, x_{i_{2s+1}}][x_{i_{2s}}, x_{i_{2s+2}}]
\cdots [x_{i_k}, x_{i_{k+1}}]-
[x_1, x_2]\cdots[x_{k}, x_{k+1}]\\
=& \sgn(\sigma)[x_{i_1}, x_{i_2}]
\cdots \left([x_{i_{2s-1}}x_{i_{2s}}, x_{i_{2s+1}}, x_{i_{2s+2}}]
-[x_{i_{2s-1}}, x_{i_{2s+1}},x_{i_{2s+2}}]x_{i_{2s}}\right.\\
&\qquad+[x_{i_{2s-1}}x_{i_{2s+1}}, x_{i_{2s}}, x_{i_{2s+2}}]
-[x_{i_{2s-1}}, x_{i_{2s}},x_{i_{2s+2}}]x_{i_{2s+1}}\left. \right)
\cdots[x_{i_k}, x_{i_{k+1}}]\\
&+\sgn(\sigma')[x_{i_1}, x_{i_2}]\cdots[x_{i_{2s-1}}, x_{i_{2s+1}}]
[x_{i_{2s}}, x_{i_{2s+2}}]\cdots [x_{i_k}, x_{i_{k+1}}]-
[x_1, x_2]\cdots[x_{k}, x_{k+1}].
\end{align*}
Therefore, from the induction hypothesis, it follows that 
\[\Psi_{k+1}(\sgn(\sigma)[x_{\sigma^{-1}(1)}, x_{\sigma^{-1}(2)}]
\cdots [x_{\sigma^{-1}(k)}, x_{\sigma^{-1}(k+1)}]-[x_1, x_2]
\cdots[x_{k}, x_{k+1}])\in \lr{\up{k}(k)}(k+1).\]
This completes the proof.
\end{proof}

Next, we give a basis of $\lr{\up{k}(k)}(k+1)$ for an 
odd $k$, which deduces that $\up{k}$ can not generated 
by $\up{k}(k)$. 

\begin{prop}
\label{yypro3.9}
Let $k\ge 3$ be an odd number and $\Theta^{k}$ the 
Specht basis of $\up{k}(k)$. Set 
\begin{align*}
B_1=& \{\tau_{k_1, \cdots, k_s}^\sigma\mid 
\sigma\in \Sigma_{k_1, \cdots, k_s}, k_1+\cdots+k_s=k+1, k_s\ge 3\}, \\
B_2=& \{\tau_{2, \cdots, 2}\ast (\sgn(\sigma)\sigma-1_{k+1})\mid 
\sigma\in \Sigma'_{2, \cdots, 2} \},\\
B_3=& \{(1_2\ucr{1}\theta_j)\ast c_i\mid \theta_j\in \Theta^{k}\}
\end{align*}
where $c_i=(1, \cdots, i-1, i+1, \cdots, k+1, i)\in \S_{k+1}$
and $\Sigma'_{2,\cdots,2}=\Sigma_{2, \cdots, 2}\backslash 
\{(2, 1, 4, 3, \cdots, k+1, k)\}$. 
Then $B_1\cup B_2 \cup B_3$ is a basis of $\lr{\up{k}(k)}(k+1)$.
Consequently,  $\up{k+1}(k+1)\not\subseteq\lr{\up{k}(k)}(k+1)$. 
\end{prop}

\begin{proof} 
Firstly, we show that $B_1\cup B_2\cup B_3\subset 
\lr{\up{k}(k)}(k+1)$. By definition, $B_3\subset 
\lr{\up{k}(k)}(k+1)$. 
Being similar to the proof of Lemma \ref{yylem3.7}, we see
that $B_1\subset \lr{\up{k}(k)}(k+1)$. By Lemma \ref{yylem3.8}, 
$B_2\subset \lr{\up{k}(k)}(k+1)$.
	
Secondly, by combining Theorem \ref{yythm1.3}
with Proposition \ref{yypro3.3}, one sees that 
$B_1\cup B_2\cup B_3$ are linearly independent.

It remains to show that each element $\theta\in \lr{\up{k}(k)}(k+1)$ 
can be linearly represented by the elements in $B_1\cup B_2\cup B_3$.
To do this, we only need to show that 
$(1_2\ucr{j} \theta)\ast\sigma$ and $(\theta\ucr{r} 1_2)\ast\sigma$ 
can be linearly represented by the elements in 
$B_1\cup B_2\cup B_3$ for all $\theta\in \Theta^{k}$, $j=1, 2$, 
$1\le r\le k$ and $\sigma\in \S_{k+1}$. Take an arbitrary element
\[\theta=\tau^\rho_{l_1,\cdots,l_m}
=\Psi_k([x_{j_{11}}, \cdots, x_{j_{1l_1}}]
\cdots[x_{j_{m1}}, \cdots, x_{j_{ml_m}}]) \]
in $\Theta^k$, where $\rho=(j_{11}, \cdots, j_{1l_1}, \cdots, j_{m1}, 
\cdots, j_{ml_m})\in \Sigma_{l_1, \cdots, l_m}$, 
$l_1+\cdots+l_m=k$, $l_i\ge 2$ $i=1, \cdots, m$. Since $k$ is odd, 
we have that $l_m\ge 3$. 

Case 1. $(1_2\ucr{1} \theta)\ast \sigma$, $\sigma\in \S_{k+1}$. 
Using the fact $\S_{k+1}=\bigcup_{i=1}^{k+1} (\S_k\times \{1\}) c_i$ 
or by Lemma \ref{yylem1.10}(1) for the ideal 
$\I=\lr{\up{k}(k)}$, one sees that 
$(1_2\ucr{1} \theta)\ast \sigma$ can be linearly represented 
by the elements in $B_3$.

Case 2. $(1_2\ucr{2} \theta)\ast \sigma$, $\sigma\in \S_{k+1}$. 
Considering $(1_2\ucr{2} \theta)\ast \sigma$ and setting  
$\varphi=(k+1,1,2,\cdots,k)$, we have
\begin{align*}
\Psi_{k+1}^{-1}&((1_2\ucr{2} \theta)\ast\varphi)
 = x_{k+1}[x_{j_{11}}, 
\cdots, x_{j_{1l_1}}]\cdots[x_{j_{m1}}, \cdots, x_{j_{ml_m}}]\\
=& ([x_{j_{11}}, \cdots, x_{j_{1l_1}}]x_{k+1}-[x_{j_{11}}, 
 \cdots, x_{j_{1l_1}}, x_{k+1}])[x_{j_{21}}, 
 \cdots, x_{j_{2l_2}}]\cdots[x_{j_{m1}}, \cdots, x_{j_{ml_m}}]\\
=& [x_{j_{11}}, \cdots, x_{j_{1l_1}}]([x_{j_{21}}, 
 \cdots, x_{j_{2l_2}}]x_{k+1}-[x_{j_{21}}, \cdots, x_{j_{2l_2}}, 
 x_{k+1}])[x_{j_{31}}, \cdots, x_{j_{3l_3}}]\cdots[x_{j_{m1}}, 
 \cdots, x_{j_{ml_m}}]\\
& -[x_{j_{11}}, \cdots, x_{j_{1l_1}}, x_{k+1}][x_{j_{21}}, 
 \cdots, x_{j_{2l_2}}]\cdots[x_{j_{t1}}, \cdots, x_{j_{tl_t}}]\\
=& [x_{j_{11}}, \cdots, x_{j_{1l_1}}]\cdots[x_{j_{t1}}, 
 \cdots, x_{j_{tl_t}}]x_{k+1}-[x_{j_{11}}, \cdots, 
 x_{j_{1l_1}}, x_{k+1}][x_{j_{21}}, \cdots, x_{j_{2l_2}}]
 \cdots[x_{j_{m1}}, \cdots, x_{j_{ml_m}}]\\
& -\cdots-[x_{j_{11}}, \cdots, x_{j_{1l_1}}]\cdots
 [x_{j_{m-1,1}}, \cdots, x_{j_{m-1,l_{m-1}}}]
 [x_{j_{m1}}, \cdots, x_{j_{ml_m}}, x_{k+1}].
\end{align*}
Therefore,  we have $(1_2\ucr{2} \theta)\ast
\varphi\in (1_2\ucr{1}\theta)+M^{k+1}_3$. This implies that 
\[(1_2\ucr{2} \theta)\ast \sigma \in 
(1_2\ucr{1}\theta)\ast \widetilde{\sigma}+M^{k+1}_3\]
where $\widetilde{\sigma}=\varphi^{-1}\sigma$.
Observe that $B_1$ is a basis of $M^{k+1}_3$. It follows 
that $(1_2\ucr{2} \theta)\ast \sigma$ can be linearly 
represented by the elements in $B_1\cup B_3$.

Case 3. $(\theta\ucr{r}  1_2)\ast \sigma, \sigma\in \S_{k+1}$. 
Considering $\theta=\tau_n$ for $n\leq k+1$ and setting 
$r=p$, by a direct computation, we have 
\begin{align*}
	\Psi_{n+1}^{-1}&(\tau_n\ucr{p} 1_2)=[x_1, \cdots, 
	x_{p-1}, x_px_{p+1}, x_{p+2} \cdots, x_{n+1}] \\  
	=&[[x_1,\cdots, x_p]x_{p+1},x_{p+2},\cdots,x_{n+1}]
	+[x_p[x_1,\cdots,x_{p-1},x_{p+1}],x_{p+2},\cdots, x_{n+1}]
	\\
	=&[[x_1,\cdots,x_p][x_{p+1},x_{p+2}],x_{p+3},\cdots,x_{n+1}]
	+ [[x_1,\cdots,x_p,x_{p+2}]x_{p+1},x_{p+3},\cdots, x_{n+1}] 
	  \\
	 &+x_p[x_1\cdots, \check{x}_p, \cdots, x_{n+1}]
	+[[x_p, x_{p+2}][x_1\cdots, \check{x}_p, x_{p+1}], x_{p+3},\cdots,x_{n+1}]\\
	 =&\cdots\cdots\\
	=& x_p[x_1, \cdots, \check{x}_p, \cdots, x_{n+1}]
	+[x_1, \cdots,\check{x}_{p+1}, \cdots, x_{n+1}]x_{p+1}\\
	&+\sum_{p+2\le i_1<\cdots<i_s\le n+1,\atop 1\le s\le n-p }[x_p, x_{i_1}, 
	\cdots, x_{i_s}][x_1, \cdots,x_{p-1}, x_{p+1}, \cdots,  
	\check{x}_{i_1}, \cdots, \check{x}_{i_s}, \cdots, x_{n+1}]\\
	&+\sum_{p+2\le i_1<\cdots<i_s\le n+1,\atop 0\le s\le n-p-1}[x_1, \cdots, 
	x_p, x_{i_1}, \cdots, x_{i_s}][x_{p+1}, \cdots,  
	\check{x}_{i_1}, \cdots, \check{x}_{i_s}, \cdots, x_{n+1}]\\
	=& -[x_1, \cdots, \check{x}_p, \cdots, x_{n+1}, x_p]
	+[x_1, \cdots, \check{x}_p, \cdots, x_{n+1}]x_p
	+[x_1, \cdots, \check{x}_{p+1}, \cdots, x_{n+1}]x_{p+1}\\
	&+\sum_{p+2\le i_1<\cdots<i_s\le n+1,\atop 1\le s\le n-p }[x_p, x_{i_1}, 
	\cdots, x_{i_s}][x_1, \cdots,x_{p-1}, x_{p+1}, \cdots,  
	\check{x}_{i_1}, \cdots, \check{x}_{i_s}, \cdots, x_{n+1}]\\
	&+\sum_{p+2\le i_1<\cdots<i_s\le n+1,\atop 0\le s\le n-p-1}[x_1, \cdots, 
	x_p, x_{i_1}, \cdots, x_{i_s}][x_{p+1}, \cdots,  
	\check{x}_{i_1}, \cdots, \check{x}_{i_s}, \cdots, x_{n+1}]
\end{align*}
where $\check{x}_i$ means that $x_i$ is omitted. Therefore,  
 we know that
\begin{align*}
\tau_{l_1, \cdots, l_m}\ucr{l_1+\cdots+l_{i-1}+p}1_2
=1_m\circ (\tau_{l_1}, \cdots, \tau_{l_{i-1}}, 
\tau_{l_i}\ucr{p} 1_2, \tau_{l_{i+1}}, \cdots, \tau_{l_m})
\end{align*}
can be linearly represented by $B_1\cup B_3$ if $1\le i< m$ 
or $i=m$, $l_m>3$. Otherwise, we have $l_1=\cdots=l_{m-1}=2, 
l_m=3$ and $i=m$. In this case, using a similar computation
as before, we have 
\begin{align*}
\Psi_k^{-1}&(\tau_{2, \cdots, 2, 3}\ucr{k-2}1_2)=[x_1,x_2]
 \cdots[x_{k-4}, x_{k-3}][x_{k-2}x_{k-1}, x_k, x_{k+1}]\\
=& [x_1,x_2]\cdots[x_{k-4}, x_{k-3}]([x_k, x_{k-1}]
 [x_{k+1}, x_{k-2}]+[x_{k}, x_{k-2}][x_{k+1}, x_{k-1}])\\
&+ [x_1,x_2]\cdots[x_{k-4}, x_{k-3}]([[x_{k+1}, x_{k-2}],
 [x_k, x_{k-1}]]+[x_{k-1}, x_k, x_{k+1}, x_{k-2}])\\
&+[x_1,x_2]\cdots[x_{k-4}, x_{k-3}]([x_{k-1}, x_k, 
 x_{k+1}]x_{k-2}+[x_{k-2}, x_k, x_{k+1}]x_{k-1})\\
\in & [x_1,x_2]\cdots[x_{k-4}, x_{k-3}]
 [x_k, x_{k-1}][x_{k+1}, x_{k-2}]-[x_1,x_2]\cdots
 [x_{k-4}, x_{k-3}][x_{k-2}, x_{k-1}][x_{k}, x_{k+1}]\\
&-(-[x_1,x_2]\cdots[x_{k-4}, x_{k-3}][x_{k}, x_{k-2}]
 [x_{k+1}, x_{k-1}]-[x_1,x_2]\cdots[x_{k-4}, x_{k-3}]
 [x_{k-2}, x_{k-1}][x_{k}, x_{k+1}])\\
&+[x_1,x_2]\cdots[x_{k-4}, x_{k-3}]([x_{k-1}, x_k, 
	x_{k+1}]x_{k-2}+[x_{k-2}, x_k, x_{k+1}]x_{k-1})+M^3_{k+1}
\end{align*}
It follows that $\tau_{2, \cdots, 2, 3}\ucr{k-2}1_2$ can 
be linearly represented by $B_1\cup B_2 \cup B_3$. 
Similarly, by a direct computation, we know that
$(\tau_{2, \cdots, 2, 3}\ucr{i}1_2)\ast \sigma$ 
can be linearly represented by $B_1\cup B_2 \cup B_3$ for 
$i=k-2, k-1, k$ and any $\sigma\in \S_{k+1}$. 

Finally, we have 
$\tau_{2, \cdots, 2}\notin \lr{\up{k}(k)}(k+1)$ and therefore
$\up{k+1}(k+1)\not\subset\lr{\up{k}(k)}(k+1)$.
\end{proof}

Combining Lemma \ref{yylem3.8} and Proposition \ref{yypro3.9},
we immediately have the following fact. 

\begin{prop}
\label{yypro3.10}
Let $\up{k}$ be the $k$-th truncation ideal of $\uas$. 
Then 
\[\gd(\up{k})=\begin{cases}
k+1, & k\ \mbox{\rm is odd};\\
k, & k\ \mbox{\rm is even}.
\end{cases}\]
\end{prop}
%


%

\subsection{Proof of Theorem \ref{yythm0.8}}
\label{yysec3.3}
Now we are ready to prove Theorem \ref{yythm0.8}.

\begin{proof}[Proof of Theorem \ref{yythm0.8}]
Let $m=\gkdim (\uas/\I)$. By Lemma \ref{yylem1.2},
$\up{m}_{\ip}=0$ where $\ip=\uas/\I$. By 
Lemma \ref{yylem1.5}, $\up{m}:=\up{m}_{\uas}
\subset \I$. 

By Proposition \ref{yypro3.10}, $\up{m}$ is generated by 
$\up{m}(m)$ if $m$ is odd, and by $\up{m}(m+1)$ if
$m$ is even. Since $\up{m}_{\iq}=0$ where $\iq=\uas/\up{m}$,
every ideal $\J$ of $\iq$ is generated by $\J(m)$ as well
as $\J(m+1)$ [Lemma \ref{yylem1.10}(3)]. So this is true for 
$\J:=\I/\up{m}$. Now the assertion follows from 
Lemma \ref{yylem1.10}(4).
\end{proof}

\section{On generalized truncation ideals}
\label{yysec4}

\subsection{Definition of a generalized truncation ideal}
\label{yysec4.1}

The notation of a generalized truncation ideal was introduced 
in \cite[Section 3.1]{BYZ}. 

\begin{Def}
\label{yydef4.1}
Let $\ip$ be a unitary operad and $M$ an $\S_m$-submodule 
of $\up{m}(m)$ for some $m$. The $\S$-submodule 
$\up{m}^M$ of $\ip$ is defined by
\begin{align} 
\notag
\up{m}^M(n)=\{\mu\in \up{m}(n)\mid \pi^I(\mu)\in M, 
\forall\; I\subset [n], \; |I|=m\}
\end{align}
for all $n\in \N$. If $\up{m}^M$ is an ideal
of $\ip$, then it is called a {\it generalized truncation ideal}
of $\ip$ associated to $M$. If further, $0\neq M\subsetneq 
\up{m}(m)$, then $M$ is called {\it admissible} and 
$\up{m}^M$ is called a {\it generalized truncation ideal of 
type I}.
\end{Def}

This definition will be extended to a more general setting
by replacing a submodule $M$ by an admissible sequence 
${\mathbb M}:=(M_{m-s},\cdots,M_{m})$, see Definition 
\ref{yydef4.10}.

Note that $\up{m}^M$ is not an operadic  ideal of $\ip$ in general. 
In fact, it is no hard to check that  $\up{m}^M$ is an 
ideal of $\ip$ if and only if the following conditions hold 
\cite[Proposition 3.2]{BYZ}:
\begin{enumerate}
\item[(1)] 
$\nu\circ x\in M$ for all $\nu\in \ip(1)$ and $x\in M$;
\item[(2)]
$x\ucr{i} \nu\in M$ for all $\nu\in \ip(1)$, $x\in M$
and $1\le i\le m$. 
\end{enumerate} 
In particular, if $\ip$ is {\it connected}, i.e. $\ip(1)=\Bbbk\1$, 
then $\up{m}^M$ is always an ideal.

\subsection{Dimension sequence and a basis of $\up{m}^M$}
\label{yysec4.2}
Clearly, we have the following fact.

\begin{lem}
\label{yylem4.2}
Let $\ip$ be a 2-unitary operad and $M$ an $\S_m$-submodule 
of $\up{m}(m)$ for some $m\ge 2$. Assume that $\up{m}^M$ is 
an ideal of $\ip$. Then $\up{m+1}\subset \up{m}^M 
\subset \up{m}$. In particular, 
\begin{align}
\label{E4.2.1}\tag{E4.2.1}
\up{m}^M=\begin{cases}
\up{m+1}, & if \ M=0, \\
\up{m}, & if \ M=\up{m}(m).
\end{cases}
\end{align}
\end{lem} 

\begin{proof} 
By the definition of $\up{m}^M$, we have 
\begin{equation}
\label{E4.2.2}\tag{E4.2.2}
\up{m}^M(n)=\begin{cases}
0, & n<m;\\
M, & n=m;\\
\{\mu\in \up{m}(n)\mid \pi^I(\mu)\in M, 
\forall I\subset [n], |I|=m\}, & n>m.
\end{cases}
\end{equation}
Clearly,  $\up{m}^M \subset \up{m}$. On the other side, 
it is obvious that $\up{m+1}(n)=0\subset \up{m}^M(n)$ for 
$n\le m$. For any $n>m$ and $\theta\in \up{m+1}(n)$, we 
have $\pi^I(\theta)=0$ for any $I\subset [n]$, $|I|=m$. 
Therefore, $\up{m+1}\subset \up{m}^M$. 

If $M=0$, by the definition of truncation ideals, we have 
$\up{m+1}=\up{m}^M$.

If $M=\up{m}(m)$, it is obvious that 
$\up{m}^M(n)=\up{m}(n)$ for all $n\le m$. For any $n>m$ 
and $\theta\in \up{m}(n)$, we have 
$\pi^I(\theta)\in \up{m}(m)$ for all $I\subset [n]$ with 
$|I|=m$ since $\up{m}$ is an ideal of $\ip$. It 
follows that $\up{m}=\up{m}^M$ if $M=\up{m}(m)$. 
\end{proof}

For a general 2-unitary operad $\ip$, one can define 
$\Lambda_{g,p}(\ip)$ to be the set of positive 
numbers $\lambda$ such that $\dim (\ip/\I)(n)
\sim \lambda n^g p^n$ where $\I$ is any operadic 
ideal of $\ip$.

\begin{lem}
\label{yylem4.3}
Let $\ip$ be a 2-unitary operad and $M$ an $\S_m$-submodule 
of $\up{m}(m)$ for some $m\ge 1$. 
\begin{enumerate}
\item[(1)]
Assume that $\up{m}^M$ is an ideal of $\ip$, and let 
$\mathcal{M}=\ip/\up{m}^M$. Then 
\begin{enumerate}
\item[(1a)] 
if $k>m$, $\up{k}_{\mathcal{M}}=0$;
\item[(1b)] 
if $k\le m$, $\up{k}_{\mathcal{M}}(n)\cong\begin{cases}
\up{k}(n), & n<m,\\
\up{k}(m)/M, & n=m,\\
\up{k}(n)/\up{m}^M(n), & n>m.
\end{cases}$
\end{enumerate}
Consequently, $\gkdim(\mathcal{M})=m+1$ if $M$ is a proper 
submodule of $\up{m}(m)$, and $\gkdim(\mathcal{M})=m$ if 
$M=\up{m}(m)$.
\item[(2)]
Suppose $\ip$ is a connected 2-unitary operad with $k$-th 
truncation ideal $\up{k}$. Let $\gamma_k(\ip):=
\dim \up{k}(k)$ and 
$T_k(\ip):=\{\dim M\mid  M 
\ \mbox{is a nonzero proper submodule of the $\Bbbk\S_k$-module  
$\up{k}(k)$}\}$. Then 
$$\Lambda_{k,1}(\ip)=\{ u/k! \mid u=\gamma_k(\ip), {\text{ or }}
u\in T_k(\ip)\}.$$
\end{enumerate}
\end{lem}

\begin{proof} (1) This follows from 
Lemmas \ref{yylem1.5} and \ref{yylem4.2}.

(2) By part (1) and Lemma \ref{yylem1.4}, 
$\Lambda_{k,1}(\ip)\supset \{ u/k! \mid u=\gamma_k(\ip), {\text{ or }}
u\in T_k\}.$

Conversely, if $\dim (\ip/\I)(n)
\sim \lambda n^k 1^n$ for all $n$, $\gkdim \ip/\I=k+1$. 
By Corollary \ref{yycor1.6} and Lemma \ref{yylem1.4},
$\lambda =\dim \up{k}_{\ip/\I}(k)=\dim \up{k}(k)/(\up{k}(k)\cap \I)$
where $\up{k}(k)\cap \I\neq \up{k}(k)$. So $\lambda 
\in \{ u/k! \mid u=\gamma_k(\ip), {\text{ or }}
u\in T_k(\ip)\}$ as desired.
\end{proof}

It is clear that $T_k(\ip):=\{\dim (\up{k}(k)/M)
\mid 0\neq \up{k}(k)/M \neq \up{k}(k)\}$. So $T_{k}(\ip)$ 
is palindrome. 

\begin{thm}	
\label{yythm4.4}
Let $\ip$ be a 2-unitary operad and $M$ an $\S_m$-submodule 
of $\up{m}(m)$ for some $m\ge 1$. Assume that $\up{m}^M$ is 
an ideal of $\ip$. Suppose that $\Theta^k$ is a 
basis for $\up{k}(k)$ for $k\ge m$ and 
$\widehat{\Theta}^m\subset \Theta^m$ is a 
basis for $M$. Denote
\[\widehat{\textbf{B}}_m(n)\colon=\{\1_2\circ 
(\theta, \1_{n-m})\ast c_I\mid \theta\in \widehat{\Theta}^m, 
I\subset [n], |I|=n-m\},\]
for any $n>m$, and 
\[\textbf{B}_k(n)\colon=\{\1_2\circ (\theta, \1_{n-k})
\ast c_I\mid \theta\in \Theta^k, I\subset [n], |I|=n-k\},\]
for all $n\ge k> m$. Then, for $n\geq m$, $\up{m}^M(n)$ has 
a $\Bbbk$-linear basis $\widehat{\textbf{B}}_m(n)\cup 
\bigcup_{m<k\le n}\textbf{B}_k(n)$.
\end{thm}

\begin{proof} If $M=\up{m}(m)$, then the assertion follows
from Theorem \ref{yythm1.3} and Lemma \ref{yylem4.2}. So 
we assume that $M$ is a proper $\S_m$-submodule of 
$\up{m}(m)$. Since $\up{m}^M(m)=M$, the assertion holds for
$n=m$. It remains to consider the case when $n>m$. 

Denote $\mathcal{M}:=\ip/\up{m}^M$. By Lemma \ref{yylem4.3}, 
we know $\gkdim \mathcal{M}=m+1$ since $M$ is a proper 
submodule of $\up{m}(m)$. Therefore, by Lemma \ref{yylem1.4}(3), 
we get
\begin{align*}
		\dim \mathcal{M}(n)
		=&\sum_{i=0}^m \gamma_\mathcal{M}(i){n\choose i}\\
		=&\sum_{i=0}^{m-1} \gamma_\ip(i){n\choose i}
		+(\gamma_\ip(m)-\dim M){n\choose m}\\
		=& \dim \frac{\ip}{\up{m+1}}(n)-(\dim M){n\choose m},
\end{align*}
since 
$$\gamma_\mathcal{M}(i)=\dim \up{i}_{\mathcal{M}}(i)
=\dim \up{i}(i)=\gamma_\ip(i)$$ 
for each $0\le i\le m-1$, and 
$$\gamma_\mathcal{M}(m)=\dim \up{m}_{\mathcal{M}}(m)
=\dim (\up{m}(m)/M)=\gamma_\ip(m)-\dim M$$ 
by Lemma \ref{yylem4.3}(1). It follows that 
\[\dim \up{m}^M(n)=\dim \up{m+1}(n)+(\dim M){n\choose m}\]
for each $n>m$. Clearly, $\widehat{\textbf{B}}_m(n)$ is a 
proper subset of $\textbf{B}_m(n)$ by 
$\widehat{\Theta}^m\subsetneq \Theta^m$. By Theorem 
\ref{yythm1.3}, we know $\bigcup_{m<k\le n}\textbf{B}_k(n)$ and 
$\bigcup_{m\le k\le n}\textbf{B}_k(n)$ are $\Bbbk$-linear bases 
for $\up{m+1}(n)$ and $\up{m}(n)$, respectively. Therefore, 
we have that  
$\widehat{\textbf{B}}_m(n)\cup \bigcup_{m<k\le n}\textbf{B}_k(n)$ 
are linearly independent. By definition, 
$\widehat{\textbf{B}}_m(n)\subset \up{m}^M(n)$ and 
$|\widehat{\textbf{B}}_m(n)|=(\dim M){n\choose m}$. From 
$\up{m+1}(n)\subset \up{m}^M(n)$, it follows that 
$\widehat{\textbf{B}}_m(n)\cup \bigcup_{m<k\le n}\textbf{B}_k(n)$ 
is a $\Bbbk$-linear basis for $\up{m}^M(n)$.
\end{proof}

Here is another way of describing generalized truncation ideals.

\begin{cor}	
\label{yycor4.5}
Let $\ip$ be a connected 2-unitary operad. Let $\I$ be an ideal 
of $\ip$ and $m\geq 1$.
\begin{enumerate}
\item[(1)]
If $\up{m+1}\subset \I\subsetneq \up{m}$, 
then $\I=\up{m}^M=\langle M\rangle +\up{m+1}$ 
for some $\S_m$-submodule $M\subsetneq \up{m}(m)$.
\item[(2)]
If $\gkdim (\ip/\I)=m+1$ and $\mdeg(\I)=m$, then 
$\I=\up{m}^M=\langle M\rangle +\up{m+1}$ 
for some $\S_m$-submodule $0\neq M\subsetneq \up{m}(m)$.
\end{enumerate}
\end{cor}

\begin{proof} (1) Since $\ip$ is connected, $\up{m}^M$ is an 
ideal for any $\S_{m}$-submodule $M\subset \up{m}(m)$. 
Now let $M=\I\cap \up{m}(m)$ which is an $\S_{m}$-submodule 
of $\up{m}(m)$. Let $\J$ be the ideal
$\langle M\rangle+\up{m+1}$. It is clear that 
$\lr{M}\subset \J\subset \I$ and 
$\lr{M}\subset \J\subset \up{m}^M$. It remains to show
that $\J=\I$ and that $\J=\up{m}^M$. 
Since $\I\cap \up{m}(m)=M$ by definition, we have $\J\cap 
\up{m}(m)=M=\I\cap \up{m}(m)$. As a consequence, 
$\I\cap \up{k}(k)=\J\cap \up{k}(k)$ for all $k\geq 1$. 
It follows from the Basis Theorem [Lemma 
\ref{yylem1.10}(1)], $\I=\J$. A similar argument shows that
$\up{m}^M=\J$. Since $\I\subsetneq \up{m}$, $M\subsetneq 
\up{m}(m)$. The assertion follows.

(2) This follows from part (1), Lemma \ref{yylem1.12},
and Corollary \ref{yycor1.6}.
\end{proof}

The next result is a generalization of Proposition \ref{yypro1.9}.

\begin{cor}	
\label{yycor4.6}
Let $\up{m}$ be the $m$-th truncation ideal of $\uas$ and $M$ 
a proper $\S_m$-submodule of $\up{m}(m)$ for some 
$m\ge 2$. Denote $\mathcal{M}=\uas/\up{m}^M$ and let
$\gamma_k$ be given in \eqref{E1.7.1}. Then 
\begin{enumerate}
\item[(1)] 
$G_{\mathcal{M}}(t)=\sum\limits_{k=0}^{m-1} \gamma_k
\dfrac{t^k}{(1-t)^{k+1}}+(\gamma_m-\dim M)
\dfrac{t^m}{(1-t)^{m+1}}$ where $\gamma_m-\dim M>0$.
\item[(2)] 
$\dim \mathcal{M}(n)=\sum\limits_{k=0}^{m-1} 
\gamma_k {n\choose k}+(\gamma_m-\dim M){n\choose m}$.
\item[(3)]
$\gkdim \mathcal{M}=m+1$.
\end{enumerate}
\end{cor}

\begin{cor}
\label{yycor4.7}
Let $\ip$ be a connected 2-unitary operad and $m\geq 2$.
\begin{enumerate}
\item[(1)]
For every maximal proper submodule $M\subsetneq \up{m}(m)$,
there is an ideal $\I_{M}$ that is maximal with respect 
to $\gkdim$ {\rm{(}}with $\gkdim (\ip/\I_{M})=m+1${\rm{)}}. 
Further, if $M\neq M'$, then $\I_{M}\neq \I_{M'}$.
\item[(2)]
If $\up{m}(m)$ has infinitely many proper maximal 
submodules, or equivalently, 
$\chi_{\up{m}(m)}=\sum_{\la\vdash m}l_\la \chi_\la$  with some 
$l_\la\geq 2$, then there are infinitely many ideals 
$\I$ of $\ip$ that are maximal with respect to
$\gkdim$ {\rm{(}}with $\gkdim (\ip/\I)=m+1${\rm{)}}.
\item[(3)]
Suppose $\ip=\uas$. For $m\ge 5$, there are infinitely 
many maximal ideals $\I$ with respect to $\gkdim$ 
{\rm{(}}with $\gkdim (\uas/\I)=m+1${\rm{)}}.
\end{enumerate}
\end{cor}

\begin{proof}
(1) Since $\ip$ is connected, $\up{m}^{M}$ is an ideal
of $\ip$ and $\gkdim \ip/\up{m}^{M}=m+1$. Since 
$\up{m+1}\subset \up{m}^M$, ${\mathcal M}:
=\ip/\up{m}^M$ is both artinian and noetherian 
\cite[Theorem 0.3]{BYZ}. So there is (at least) one 
ideal, denoted by $\I_{M}$, that is maximal 
with respect to to $\gkdim$ {\rm{(}}with 
$\gkdim (\ip/\I_{M})=m+1${\rm{)}}. Since $\up{m}
\not\subset \I_{M}$ [Corollary \ref{yycor1.6}], $\I_{M}
\cap \up{m}(m)=M$. Hence $\I_{M}\neq \I_{M'}$ 
if $M\neq M'$. 

(2) This follows from part (1) and the hypothesis that 
$\up{m}(M)$ has infinitely many proper maximal submodules. 

(3) 
Suppose that $\chi_{\up{m}(m)}=\sum_{\la\vdash m}l_\la \chi_\la$ 
is the decomposition of the character $\chi_{\up{n}(n)}$ 
as an $\S_n$-representation, where $\chi_\la$ is the 
irreducible character corresponding to the partition $\la$ 
of $n$. By \cite[Proposition 4.7]{JSV}, $l_{(m-2, 2)}=m-3\ge 2$.
It follows that $\up{m}(m)$ has infinitely many proper 
maximal submodules. The assertion follows from part (2). 
\end{proof}

An equivalent result of Corollary \ref{yycor4.7}(3) on minimal 
varieties were obtained in \cite{GMZ}. 

The next corollary also follows from Lemma \ref{yylem4.3}.

\begin{cor}
\label{yycor4.8}
Suppose that $M$ is a proper submodule of the right 
$\Bbbk\S_m$-module $\up{m}(m)\subset \uas(m)$ for $m\geq 2$. 
If $\up{m}^M$ is maximal with respect to $\gkdim$
{\rm{(}}with $\gkdim \uas/\up{m}^M=m+1${\rm{)}}, 
then $M$ is a maximal submodule of $\up{m}(m)$. 
\end{cor}

\subsection{The generators of $\up{m}^M$}
\label{yysec4.3}
In this subsection, we study the generators of the 
ideal $\up{m}^M$ in $\uas$. 

\begin{prop}
\label{yypro4.9}
Let $\ip=\uas$ and let $M\neq 0$ be a proper submodule of 
$\up{m}(m)$ for some $m\ge 2$. 
\begin{enumerate}
\item[(1)]
Then $\up{m}^M=\langle \up{m}^M(m+1)\rangle$ if 
$m$ is odd, and $\up{m}^M=\langle \up{m}^M(m+2)\rangle$ if 
$m$ is even. Equivalently,
\begin{enumerate}
\item[(1a)] $m\le \gd(\up{m}^M)\le m+1$, if $m$ is odd;
\item[(1b)] $m\le \gd(\up{m}^M)\le m+2$, if $m$ is even.
\end{enumerate}
\item[(2)]
Suppose that $\zeta$ is a 
generator of $M$ as a right $\Bbbk\S_m$-module, and $\beta_{i}$ 
is a generator of the right $\Bbbk\S_{i}$-module $\up{i}(i)$ for 
$i=m+1, m+2$.
\begin{enumerate}
\item[(2a)] 
If $\gd(\up{m}^M)=m$, then $\up{m}^M=\lr{\zeta}$;
\item[(2b)] 
If $\gd(\up{m}^M)=m+1$, then $\up{m}^M=\lr{1_2\circ 
(\zeta, 1_1)+\beta_{m+1}}$; and 
\item[(2c)] 
If $\gd(\up{m}^M)=m+2$, then $\up{m}^M=\lr{1_2\circ 
(\zeta, 1_2)+1_2\circ (\beta_{m+1}, 1_1)+\beta_{m+2}}$. 
\end{enumerate}
\end{enumerate}
\end{prop}

\begin{proof} 
(1) The upper bounds follow from Theorem \ref{yythm0.8}
and Corollary \ref{yycor4.6}(3). 
Let $\I= \up{m}^M$. By definition $\I\subset \up{m}$. So 
$\I(m-1)=0$. The lower bounds follows.

(2) This follows from part (1) and Lemma \ref{yylem1.10}(2).
\end{proof}

For the rest of the section, we study generalized truncation
ideals associated to an admissible sequence.

\subsection{Extended definition of a generalized truncation 
ideal}
\label{yysec4.4}
For simplicity we only consider connected 2-unitary operads 
here. Throughout the rest of this section, let $\ip$ be a 
connected 2-unitary operad, $m\geq 1$, and let $s<m$ be a nonnegative 
integer. Let $M_i$ be a $\Bbbk \S_i$-submodule of $\up{i}_\ip(i)$ 
for $i=m-s,\cdots,m$ and let ${\mathbb M}$ denote the sequence
$(M_{m-s},M_{m-s+1}, \cdots, M_{m})$.

\begin{Def}
\label{yydef4.10}
Retain the above notation. We say that 
${\mathbb M}:=(M_{m-s}, M_{m-s+1}, \cdots, M_{m})$ is an 
{\it $m$-admissible sequence} if $M_{m-s}\neq 0$, 
$M_{m}\neq \up{m}(m)$ and 
\[(\lr{M_{m-s}}(j)+\cdots+\lr{M_{j-1}}(j))\cap \up{j}(j)
\subset M_j\] 
for all $j=m-s+1, \cdots, m$. And the ideal 
$\sum_{0\le i\le s}\lr{M_{m-i}}+\up{k+1}$ is called the 
{\it generalized truncation ideal associated to 
${\mathbb M}$}, denoted by $\upa{m}^{\mathbb M}$. 
\end{Def}

\begin{rmk}
\label{yyrem4.11}
If $s=0$, ${\mathbb M}$ is a singleton $\{M\}$, and the above 
definition recovers the definition of a generalized truncation 
ideal of type I given in Definition \ref{yydef4.1}, see 
Corollary \ref{yycor4.5}. As in Definition \ref{yydef4.1}, 
this generalized truncation ideal is denoted by $\up{m}^M$ and 
$M$ is called {\it admissible}. 

If $s=1$, then an $m$-admissible sequence ${\mathbb M}:
=(M_{m-1},M_m)$ is also called an {\it $m$-admissible pair} 
and $\up{m}^{\mathbb M}$ is called a {\it generalized 
truncation ideal of type II}.
\end{rmk}

\subsection{Basic properties of $\upa{m}^{\mathbb M}$}
\label{yysec4.5}
In this subsection we briefly discuss some basic properties 
of generalized truncation ideals $\upa{m}^{\mathbb M}$. Some 
non-essential details are omitted.

\begin{lem}
\label{yylem4.12}
Let ${\mathbb M}:=(M_{m-s}, \cdots, M_m)$ be an 
$m$-admissible sequence for $m\ge 2$ and 
$\up{m}^{\mathbb{M}}$ the associated generalized 
truncation ideal. Then
\begin{enumerate}
\item[(1)] 
$\up{m}^{\mathbb{M}}(j)\cap \up{j}(j)=M_j$ for all 
$j=m-s, \cdots, m$.
\item[(2)] 
$\gkdim(\up{m}^\mathbb{M})=m+1$.
\end{enumerate}
\end{lem}

\begin{proof} 
(1) Denote $\J_j=\sum_{m-s\le i\le j-1}\lr{M_i}$ for 
$j=m-s, \cdots, m$. So $\J_{m-s}=0$. By definition, 
we have $\J_j(j)\cap \up{j}(j)\subseteq M_j$. For 
any $m-s\le j\le m$, we have 
$$\langle M_j \rangle \subset \upa{m}^\mathbb{M} \subset 
\J_{j}+ \langle M_{j}\rangle +\up{j+1}=\J_j+
\up{j}^{M_{j}}.$$
It is clear that $(\langle M_j \rangle\cap \up{j})(j)=M_j$. 
It remains to show that $((\J_j+\up{j}^{M_{j}}) 
\cap \up{j})(j)=M_j$. By Theorem \ref{yythm4.4},
$\up{j}^{M_{j}} \cap \up{j}(j)=M_j$. Since $\J_j
\cap \up{j}(j)\subseteq M_j$, we obtain that 
$$ (\J_j+\up{j}^{M_{j}}) \cap \up{j}(j)
\subseteq (\J_{j} \cap M_j)+M_j=M_j$$
as required. 

(2) By Part (1), we know $\up{m}^\mathbb{M}(m)\cap 
\up{m}(m)=M_m\subsetneq \up{m}(m)$. Clearly, 
$\up{m}(m)\not\subset \up{m}^\mathbb{M}(m)$ and 
$\up{m}\not\subset \up{m}^\mathbb{M}$. By definition, 
we have $\up{m+1}\subset \up{m}^\mathbb{M}$. From 
Corollary \ref{yycor1.6}, it follows that 
$\gkdim(\ip/\up{m}^\mathbb{M})=m+1$.
\end{proof}

\begin{exm}
\label{yyexa4.13} 
If $\ip=\uas$. We claim that the only $4$-admissible pair is 
${\mathbb M}=(\up{3}(3), \langle \up{3}(3)\rangle \cap \up{4}(4))$.
As a consequence, there is a unique generalized truncation ideal
$\I$ of type II with $\gkdim (\uas/\I)=5$, which is $\upa{4}^{\mathbb M}$ 
(and also equals to $\lr{\up{3}(3)}+\up{5}$).

Let ${\mathbb M}=(M_3,M_4)$ be any $4$-admissible pair. Then 
$0\neq M_3\subset \up{3}(3)$. Since $\up{3}(3)$ is a 2-dimensional
simple module, $M_3=\up{3}(3)$ which is generated by $\tau_{3}$,
see \eqref{E1.0.2}. Let $N=\lr{M_3}(4)\cap \up{4}(4)$.
Since $\Lie(n)$ is generated by Dynkin elements \eqref{E1.0.1},
$\Lie(4)\subset N$. By Lemma \ref{yylem3.8}, $\tau_{2,2}\ast 
(\sgn(\sigma)\sigma-1_{4})\in N$. It follows from Lemma \ref{yylem3.1}
that $\up{4}(4)/N$ has dimension 0 or 1 (generated by 
$\tau_{2,2}$ as a $\Bbbk$-vector space). By the proof of 
Proposition \ref{yypro3.9}, 
$\tau_{2,2}\not\in\lr{\up{3}(3)}$,
so it is not in $N$. This implies that $\dim \up{4}(4)/N=1$. 
By definition, $M_4\subsetneq \up{4}(4)$ and $N \subset M_4$. 
These imply that there is only one choice for $M_4$, which is 
$N=\lr{M_3}(4)\cap \up{4}(4)$. 

It remains to show that ${\mathbb M}:
=(\up{3}(3), \langle \up{3}(3)\rangle \cap \up{4}(4))$ is indeed a 
$4$-admissible pair. By the previous paragraph, it suffices to show
that $\gkdim \uas/\I=5$ where $\I=\lr{\up{3}(3)}+\up{5}$.
Let $\ip=\uas/\I$. Then the last paragraph implies that
$$\up{k}_{\ip}(k)=\begin{cases}
 \up{4}(4)/N\cong \Bbbk , & k=4,\\
0 & k\geq 5.
\end{cases}
$$
Therefore $\gkdim \uas/\I=5$ by Corollary \ref{yycor1.6}. 
The assertion follows.
\end{exm}

Similar to Corollary \ref{yycor4.6} we have the following.
Its proof is obvious.

\begin{lem}	
\label{yylem4.14}
Let $\ip=\uas$.  Suppose that ${\mathbb M}:=(M_{m-s},\cdots, M_{m})$ is an $m$-admissible 
sequence and $\upa{m}^{\mathbb M}$ is the associated generalized
ideal. Denote $\mathcal{M}=\uas/\upa{m}^{\mathbb M}$ 
and let $\gamma_k$ be given in \eqref{E1.7.1}. Then 
\begin{enumerate}
\item[(1)] 
$G_{\mathcal{M}}(t)=\sum\limits_{k=0}^{m-s-1} \gamma_k
\dfrac{t^k}{(1-t)^{k+1}}+
\sum\limits_{k=m-s}^{m} (\gamma_{k}-\dim M_{k})
\dfrac{t^{k}}{(1-t)^{k+1}}$,  where $\gamma_m-\dim M_m>0$.
\item[(2)] 
$\dim \mathcal{M}(n)=\sum\limits_{k=0}^{m-s-1} 
\gamma_k {n\choose k}+\sum\limits_{k=m-s}^{m} 
(\gamma_{k}-\dim M_{k}){n\choose k}$.
\end{enumerate}
\end{lem} 

The next lemma is similar to Proposition \ref{yypro4.9}.

\begin{lem}
\label{yylem4.15}
Let $\ip=\uas$. Suppose that ${\mathbb M}:=(M_{m-s},\cdots,M_{m})$ is an 
$m$-admissible sequence  for some $m\ge 3$ and $s\geq 1$, and 
$\upa{m}^{\mathbb M}$ be the associated generalized truncation 
ideal. 
\begin{enumerate}
\item[(1)]
Then $\upa{m}^{\mathbb M}=\langle \upa{m}^{\mathbb M}(m+1)\rangle$ if 
$m$ is odd, and $\upa{m}^{\mathbb M}=\langle \upa{m}^{\mathbb M}(m+2)\rangle$ if 
$m$ is even. Equivalently,
\begin{enumerate}
\item[(1a)] $m-s\le \gd(\upa{m}^{\mathbb M})\le m+1$, if $m$ is odd;
\item[(1b)] $m-s\le \gd(\upa{m}^{\mathbb M})\le m+2$, if $m$ is even.
\end{enumerate}
\item[(2)]
Suppose that $\zeta_{j}$ is a generator of $M_{j}$ as a 
right $\Bbbk\S_{j}$-module for $j=m-s,\cdots,m$, and 
$\beta_{i}$ is a generator of the right $\Bbbk\S_{i}$-module 
$\up{i}(i)$ for $i=m+1, m+2$.
\begin{enumerate}
\item[(2a)] 
If $\gd(\upa{m}^{\mathbb M})=w\leq m$, then $\upa{m}^{\mathbb M}=
\lr{\sum_{j=m-s}^{w} 1_{w-j+1} \ucr{1}\zeta_{j}}$;
\item[(2b)] 
If $\gd(\upa{m}^{\mathbb M})=m+1$, then $\upa{m}^{\mathbb M}=
\lr{\sum_{j=m-s}^{m} 1_{m-j+2} \ucr{1}\zeta_{j}+\beta_{m+1}}$; and 
\item[(2c)] 
If $\gd(\upa{m}^{\mathbb M})=m+2$, then $\upa{m}^{\mathbb M}=
\lr{\sum_{j=m-s}^{m} 1_{m-j+3} \ucr{1}\zeta_{j}+1_{2}\ucr{1} \beta_{m+1}
+\beta_{m+2}}$.
\end{enumerate}
\end{enumerate}
\end{lem}

The proof of above is similar to the proof of 
Proposition \ref{yypro4.9}, so it is omitted.
Similar to Corollary \ref{yycor4.5} we have the
following.

\begin{thm}
\label{yythm4.16}
Let $\ip$ be a connected 2-unitary operad and $m\ge 1$. 
\begin{enumerate}
\item[(1)] 
Let $\J$ be an ideal of $\ip$ such that 
$\gkdim (\ip/\J) =m+1$ and $\mdeg(\J)=m-s$ for some 
$s\geq 0$. Let $M_{j}$ be $\J(j)\cap \up{j}(j)$ for 
all $j$. Then the following hold.
\begin{enumerate}
\item[(1a)]
${\mathbb M}:=(M_{m-s},\cdots, M_m)$ is an $m$-admissible 
sequence.
\item[(1b)]
$\J=\up{m}^{\mathbb M}$.
\item[(1c)]
There is a unique ideal $\I$ of $\ip$ with
$\mdeg(\I)=m-s$ and $\gkdim(\ip/\I)=m+1$ such that 
$\I(j)\cap \up{j}(j)=M_j$ for all $j=m-s,\cdots,m$.
\end{enumerate}
\item[(2)]
Given any $m$-admissible sequence ${\mathbb M}'
=(M'_{m-s},\cdots,M'_{m})$, 
there is a unique ideal $\I$ $($which is $\up{m}^{{\mathbb M}'})$ 
of $\ip$ with $\mdeg(\I)=m-s$ and $\gkdim(\ip/\I)=m+1$ 
such that $\I(j)\cap \up{j}(j)=M'_j$ for all $j=m-s,\cdots,m$.
\item[(3)]
There is a one-to-one correspondence between the set of 
$m$-admissible sequences and the set of ideals $\J\neq 
\up{m+1}$ such that $\gkdim (\ip/\J)=m+1$.
\end{enumerate}
\end{thm}

\begin{proof}
(1) Since $\mdeg(\J)=m-s$, $M_{m-s}=\J(m-s)\neq 0$. 
We claim that $M_m\subsetneq \up{m}(m)$. Suppose 
to the contrary that that $M_m=\up{m}(m)$. Then 
$\up{m}=\up{m}^{M_m} =\lr{M_m}+\up{m+1}\subset \J(m)$. 
This contradicts $\gkdim(\ip/\J)=m+1$ from 
Corollary \ref{yycor1.6}. Therefore the claim is proved.
Furthermore, it is easily seen that 
\[\left(\sum_{m-s\le i\le j-1}\lr{M_i}(j)\right)\cap \up{j}(j)
\subset \J(j)\cap \up{j}(j)=M_j.\]
By Definition \ref{yydef4.10}, ${\mathbb M}$ is an 
$m$-admissible sequence. Thus (1a) is proved. 

By the above proof and Lemma \ref{yylem4.12}, 
$\up{m}^\mathbb{M}\cap \up{j}(j)=\J\cap \up{j}(j)$
for all $j=m-s,\cdots,m$. It is obvious that 
$\up{m}^\mathbb{M}\cap \up{j}(j)=\J\cap \up{j}(j)$ for 
$j<m-s$ and $j>m$. By the Basis Theorem 
[Lemma \ref{yylem1.10}(1)], 
$\J$ and $\up{m}^\mathbb{M}$ have the same generating 
series. Since $\up{m}^\mathbb{M}\subseteq \J$, we 
have $\up{m}^\mathbb{M}=\J$. Thus (1b) is proved. 
Finally part (1c) follows from part (1b).

(2) Let $\J=\up{m}^{{\mathbb M}'}$. The assertion 
follows from part (1c).

(3) Given an $m$-admissible sequence ${\mathbb M}:=
(M_{m-s},\cdots, M_m)$, $\up{m}^{\mathbb M}$ is a 
generalized truncation ideal associated to ${\mathbb M}$.
By construction, $\gkdim (\ip/\up{m}^{\mathbb M})=m+1$ and
$\up{m}^{\mathbb M}\neq \up{m+1}$ since ${\mathbb M}$ 
is admissible. Conversely, given an ideal $\J\neq \up{m+1}$ 
with $\gkdim (\ip/\J)=m+1$, let $M_{j}=\J(j)\cap \up{j}(j)$ 
for all $j$. Then ${\mathbb M}:=(M_{m-s},\cdots, M_m)$, 
where $m-s=\mdeg(\J)$, is an $m$-admissible sequence by 
part (1a). By parts (1,2), these give rise to a 
one-to-one correspondence. 
\end{proof}

By the above theorem, the classification of ideals $\J\subseteq \uas$ 
with $\gkdim (\uas/\J)=m+1$ is equivalent to the classification of 
$m$-admissible sequences. 

\section{$\uas/\I$ of low GK-dimension and the related PI-algebras}
\label{yysec5}

The quotient operads of $\uas$ of GK-dimension less than $5$ 
have been classified in \cite[Theorem 0.6]{BYZ}. In \cite{OV}, 
the authors obtained a complete list of PI-algebra varieties 
with polynomial growth $n^4$. In Subsection \ref{yysec5.1}, 
we will give a new proof of an equivalent result using 
ideals of $\uas$. As an application, we further 
compute the generating identities for the PI-algebras with 
polynomial $n^k$ for $k\le 4$.

\subsection{$\uas/\I$ of GK-dimension 5}
\label{yysec5.1}
By a direct computation, we have the following facts on the 
truncation ideals $\up{3}$ and $\up{4}$. 

\begin{lem}
\label{yylem5.1}
Let $\up{k}$ the $k$-th truncation ideal of $\uas$. Then the 
following hold.
\begin{itemize}
\item[(1)] 
$\up{3}(3)=\Lie(3)$ is a $2$-dimensional simple 
$\Bbbk \S_3$-module, and the character 
$\chi_{\up{3}(3)}=\chi_{(2, 1)}$.
\item[(2)] 
$\up{4}(4)$ is a $9$-dimensional right module over 
$\Bbbk\S_4$, and the character
\[\chi_{\up{4}(4)}=\chi_{(1^4)}+\chi_{(2^2)}
+\chi_{(2, 1^2)}+\chi_{(3, 1)}.\]
\item[(3)] 
$\up{3}(4)$ is a $17$-dimensional right module over 
$\Bbbk \S_4$, and the character
\[\chi_{\up{3}(4)}=\chi_{(1^4)}+2\chi_{(2^2)}
+2\chi_{(2, 1^2)}+2\chi_{(3, 1)}.\]
\end{itemize}
\end{lem}

By Lemma \ref{yylem5.1}(2) and (3), we can assume the irreducible 
decompositions 
\begin{align}
\label{E5.1.1}\tag{E5.1.1}
\up{4}(4)=V_{(1^4)}\oplus V_{(2^2)}\oplus V_{(2, 1^2)}
\oplus V_{(3, 1)},
\end{align}
and 
\begin{equation}
\label{E5.1.2}\tag{E5.1.2}
\up{3}(4) = V_{(1^4)}\oplus V_{(2^2)}\oplus W_{(2^2)} 
\oplus V_{(2, 1^2)}\oplus W_{(2, 1^2)}\oplus V_{(3, 1)}
\oplus W_{(3, 1)}
\end{equation}
where $V_\la$ (and $W_\la\cong V_\la$) are the irreducible 
components with the character $\chi_\la$ for $\la\vdash 4$.

Now we are ready to prove both Theorems \ref{yythm0.4} and 
\ref{yythm0.6} as follow.

\begin{thm}
\label{yythm5.2}
Let $\mathcal{A}_\I:=\uas/\I$ be a quotient operad of 
$\uas$ of GK-dimension $5$. Then either $\I(3)=\up{3}(3)$ or 
$\I(3)=0$ and the following hold.
\begin{enumerate}
\item[(1)] 
If $\I(3)=\up{3}(3)$, then $\I$ is $\up{4}^{\mathbb M}$ 
where ${\mathbb M}$ is the unique $4$-admissible pair 
given in Example {\rm{\ref{yyexa4.13}}}. Further, the 
codimension sequence is 
\[\dim(\mathcal{A}_\I(n))=1+{n\choose 2}+{n\choose 4},
\ for\ all  \ n\in \N.\]
\item[(2)] 
If $\I(3)=0$, then $\I$ is a generalized truncation 
ideal $\up{4}^{M}$ where $M:=\I(4)$ is a proper submodule 
of $\up{4}(4)$, 
and	the codimension sequences are
\[\dim(\mathcal{A}_\I(n))=1+{n\choose 2}+2{n\choose 3}
+(9-\dim \I(4)){n\choose 4}, \ for\  all \ n\in \N.\]
\end{enumerate}
\end{thm}

\begin{proof}
It is clear that $\I(k)=0$ for all $k\leq 2$. As a 
consequence, $\mdeg(\I)\geq 3$ and $\I\subset \up{3}$. 
Observe that $\up{3}(3)$ is a 2-dimensional simple 
$\Bbbk \S_3$-module. So there are only two cases, namely, 
either $\I(3)=0$ or $\I(3)=\up{3}(3)$. Since 
$\gkdim (\uas/\I)=5$, by Corollary \ref{yycor1.6}, 
$\up{5}\subset \I$ and $\up{4}\not\subset \I$.

(1) Case 1: $\I(3)=\up{3}(3)$. Since $\up{4} \not\supset \I$,
$\I$ is not a generalized truncation ideal of type I. Since 
$\up{5}\subset \I \subset \up{3}$, $\I$ is a generalized 
truncation ideal $\up{4}^{\mathbb M}$ for some $4$-admissible 
pair ${\mathbb M}$ [Theorem \ref{yythm4.16}]. So $\I$ is 
given in Example \ref{yyexa4.13} and $\I=\lr{\up{3}(3)}+\up{5}$. 
Now $M_3=\I(3)=\up{3}(3)$ and $M_4=\langle \up{3}(3)\rangle 
\cap \up{4}(4)$. By Remark \ref{yyrem1.1} (then 
$\I=\I_{E}+\up{5}$) or a direct computation, one has that
$\dim \up{4}(4)/M_4=1$. Hence codimension sequence formula 
follows from Lemma \ref{yylem4.14}(1).

(2) Case 2: $\I(3)=0$. Then $\up{5}\subset \I \subset \up{4}$, 
consequently, $\I$ is either $\up{5}$ or a generalized 
truncation ideal of type I [Corollary \ref{yycor4.5}]. The 
codimension sequence follows from Corollary \ref{yycor4.6}(1).
\end{proof}

\begin{rmk}
\label{yyrem5.3} 
It is worth mentioning that $\up{5}(6)\not\subset \T_3(6)$
where $\T_3$ is the ideal $\lr{\up{3}(3)}$ in $\uas$. By 
\cite[Theorem 2.16]{BYZ}, we know that $\T_3(6)$ is 
a $\Bbbk\S_6$-module generated by the following elements
\begin{itemize}
\item 
$\theta\circ (1_{k_1}, 1_{k_2}, 1_{k_3})$ with $k_1+k_2+k_3=6$, 
$k_i\ge 1$, 
\item 
$1_2\circ (\theta\circ (1_{k_1}, 1_{k_2}, 1_{k_3}), 1_1)$,
$1_2\circ (1_1, \theta\circ (1_{k_1}, 1_{k_2}, 1_{k_3}))$ 
with  $k_1+k_2+k_3=5$, $k_i\ge 1$, 
\item 
$1_2\circ (\theta\circ (1_{k_1}, 1_{k_2}, 1_{k_3}), 1_2)$,
$1_2\circ (1_2, \theta\circ (1_{k_1}, 1_{k_2}, 1_{k_3}))$, 
$1_3\circ (1_1, \theta\circ(1_{k_1}, 1_{k_2}, 1_{k_3}), 1_1)$ 
with $k_1+k_2+k_3=4$, $k_i\ge 1$, 
\item 
$1_2\circ (\theta, 1_3)$, $1_2\circ (1_3, \theta)$, 
\item $1_3\circ (1_2, \theta,  1_1)$, $1_3\circ (1_1, \theta,  1_2)$  
\end{itemize}
for all $\theta\in \up{3}(3)$, where $1_n$ is the identity 
element of $\S_n$. By a computation assisted by a 
computer, we have $\dim\T_3(6)=688$. Note that we have 
$\dim \frac{\uas}{\T_3+\up{5}}(6)=31$ by Theorem 
\ref{yythm5.2}(1) and 
hence $\dim (\T_3(6)+\up{5}(6))=689$. Therefore, 
$\T_3(6)\subsetneq \T_3(6)+\up{5}(6)$ and 
$\up{5}(6)\not\subset \T_3(6)$. 
\end{rmk} 

Next, we give more details about ideals $\I$ of 
$\uas$ such that $\gkdim(\uas/\I)=5$.

By Theorem \ref{yythm5.2}(2), each quotient operad 
$\uas/\I$ with $\I(3)=0$ 
is of form ${\uas}/{\up{4}^M}$. Observe that there are 15 
kinds of proper submodules of $\up{4}(4)$ since any two 
irreducible components are not isomorphic. There are 15 different
kinds of quotient operads $\uas/\I$ with $\I(3)=0$ of GK-dimension 5. 
By Theorem \ref{yythm5.2}(2), there is a unique ideal 
$\I$ with $\I(3)\neq 0$ 
($\I(3)=\up{3}(3)$) such that $\gkdim(\uas/\I)=5$. Therefore, 
there are total 16 quotient operads $\uas/\I$ of GK-dimension 5.
Here is an operadic version of a result in \cite{GMZ}.

\begin{cor}
\label{yycor5.4}
Let $\I$ be an ideal of $\uas$ and $\gkdim\uas/\I=5$.
Then there are $4$ maximal ideals $\I$ with respect to 
$\gkdim$, namely,  
\begin{enumerate}
\item[(1)] 
$\up{4}^{M_1}$ with 
$\chi_{M_1}=\chi_{(1^4)}+\chi_{(2, 1^2)}+\chi_{(3, 1)}$,
\item[(2)] 
$\up{4}^{M_2}$ with 
$\chi_{M_2}=\chi_{(1^4)}+\chi_{(2^2)}+\chi_{(3, 1)}$,
\item[(3)] 
$\up{4}^{M_3}$ with 
$\chi_{M_3}=\chi_{(1^4)}+\chi_{(2^2)}+\chi_{(2, 1^2)}$,
\item[(4)] 
$\upa{4}^{\mathbb M}=\T_3+\up{5}$.
\end{enumerate}
\end{cor}

\begin{proof} 
By Theorem \ref{yythm5.2} and Corollary \ref{yycor4.8}, 
it is easily seen that $\up{4}^M$ is not maximal with 
respect to $\gkdim$ if $M$ has at most two 
irreducible components. Moreover, by a direct computation,
$\up{4}^{M_4}$ is not maximal with respect to 
$\gkdim$ since $\up{4}^{M_4} \subsetneq \T_3+\up{5}$, 
where $\chi_{M_4}=\chi_{(2^2)}+\chi_{(2, 1^2)}+\chi_{(3, 1)}$, 
while $\up{4}^{M_i}\not\subset \T_3+\up{5}$ for $i=1, 2, 3$. 
It follows that $\up{4}^{M_i}$, for $i=1, 2, 3$, and 
$\T_3+\up{5}$ are maximal with respect to $\gkdim$.
\end{proof}

From decomposition \eqref{E5.1.1}, we observe that any two 
distinct irreducible components of $\up{4}(4)$ are not 
isomorphic. From \cite[Theorem 3.6.2]{We}, we can obtain 
a basis of the component $V_{\la}$ by computing the image 
of the linear map 
\[\varphi_{\lambda}=\sum_{\sigma\in \S_4} 
\chi_{\lambda}(\sigma)\rho_\sigma\colon \up{4}(4)\to \up{4}(4)\] 
for each $\la\vdash 4$, where $\rho_\sigma$ is the right 
multiplication transformation given by $\sigma$. 
Furthermore, using the Specht basis of $\up{4}(4)$
\begin{center}
\begin{tabular}{lll}
$\theta_1=\tau_{2, 2}\ast (2, 1, 4, 3)$,  & 
$\theta_2=\tau_{2, 2}\ast (3, 1, 4, 2)$, & 
$\theta_3=\tau_{2, 2}\ast (3, 2, 4, 1)$,\\
$\theta_4=\tau_4\ast(4, 1, 2, 3)$, & 
$\theta_5=\tau_4\ast(4, 1, 3, 2)$, & 
$\theta_6=\tau_4\ast(4, 2, 1, 3)$,\\
$\theta_7=\tau_4\ast(4, 2, 3, 1)$, & 
$\theta_8=\tau_4\ast(4, 3, 1, 2)$, & 
$\theta_9=\tau_4\ast(4, 3, 2, 1)$,
\end{tabular}
\end{center}
we obtain the generator $\zeta_{\la}$ of each 
irreducible component $V_{\la}$ as follows
\begin{equation}
\label{E5.4.1}\tag{E5.4.1}
\text{Table}
\end{equation}
\begin{center}
\begin{tabular}{|c|c|}
\hline 
Irreducible Component & Generator\\ \hline
$V_{(1^4)}$ & $\zeta_{(1^4)}=2\theta_1-2\theta_2
+2\theta_3-\theta_4+\theta_5+\theta_6-\theta_7-\theta_8+\theta_9$ \\ \hline
$V_{(2^2)}$ & $\zeta_{(2^2)}=2\theta_1+4\theta_2
+2\theta_3-\theta_4+\theta_5-2\theta_6+2\theta_7
-\theta_8+\theta_9$ \\ \hline
$V_{(3, 1)}$ & $\zeta_{(3, 1)}=\theta_4+3\theta_5
-2\theta_6+2\theta_7-3\theta_8-\theta_9$ \\ \hline
$V_{(2, 1^2)}$ &
$\zeta_{(2, 1^2)}=\theta_4-\theta_5-\theta_6
+\theta_7+\theta_8-\theta_9$ \\ \hline
\end{tabular}
\end{center}

Using this table and Proposition \ref{yypro4.9}, 
we obtain a generator for each $\up{4}^M$ as follows:

\begin{prop}
\label{yypro5.5}
\begin{enumerate}
\item[(1)]
Here is a list of generators of $\up{4}^M$:
\begin{equation}
\label{E5.5.1}\tag{E5.5.1}
\text{Table}
\end{equation}
\begin{center}
\begin{tabular}{|c|c|c|}
\hline
$M$ & Generator  & $\gd(\up{4}^M)$ \\ \hline 
$V_{(1^4)}+V_{(3,1)}+V_{(2,1^2)}$ & 
$\zeta_{(1^4)}+\zeta_{(3, 1)}+\zeta_{(2, 1^2)}$ & $4$ \\ \hline
$V_{(1^4)}+V_{(2^2)}+V_{(3,1)}$   & 
$\zeta_{(1^4)}+\zeta_{(2^2)}+\zeta_{(3, 1)}$    & $4$ \\ \hline
$V_{(1^4)}+V_{(3,1)}$             &  
$\zeta_{(1^4)}+\zeta_{(3,1)}$    & $4$ \\ \hline
$V_{(1^4)}+V_{(2^2)}+V_{(2,1^2)}$ & 
$1_2\circ(\zeta_{(1^4)}+\zeta_{(2^2)}+\zeta_{(2,1^2)}, 1_1)+\beta_5$  & $5$ \\ \hline 
$V_{(1^4)}+V_{(2,1^2)}$           &
$1_2\circ (\zeta_{(1^4)}+\zeta_{(2,1^2)}, 1_1)+\beta_5$    & $5$ \\ \hline
$V_{(1^4)}+V_{(2^2)}$             &    
$1_2\circ(\zeta_{(1^4)}+\zeta_{(2^2)}, 1_1)+\beta_5$   & $5$ \\ \hline
$V_{(1^4)}$                       &    
$1_2\circ(\zeta_{(1^4)}, 1_1) +\beta_5$         & $5$ \\ \hline
$V_{(2^2)}+V_{(3,1)}+V_{(2,1^2)}$ &  
$1_2\circ (\zeta_{(2^2)}+\zeta_{(3,1)}+\zeta_{(2,1^2)}, 1_2)
+1_2\circ (\beta_5, 1_1)+\beta_6$ & $6$ \\ \hline
$V_{(2^2)}+V_{(2,1^2)}$          &  
$1_2\circ (\zeta_{(2^2)}+\zeta_{(2,1^2)}, 1_2)
+1_2\circ (\beta_5, 1_1)+\beta_6$  & $6$ \\ \hline
$V_{(3,1)}+V_{(2,1^2)}$          & 
$1_2\circ (\zeta_{(3,1)}+\zeta_{(2,1^2)}, 1_2)
+1_2\circ (\beta_5, 1_1)+\beta_6$   & $6$ \\ \hline
$V_{(2^2)}+V_{(3,1)}$            &
$1_2\circ (\zeta_{(2^2)}+\zeta_{(3,1)}, 1_2)
+1_2\circ (\beta_5, 1_1)+\beta_6$   & $6$ \\ \hline
$V_{(3,1)}$                      & 
$1_2\circ (\zeta_{(3,1)}, 1_2)
+1_2\circ (\beta_5, 1_1)+\beta_6$       & $6$ \\ \hline
$V_{(2,1^2)}$  & $1_2\circ (\zeta_{(2,1^2)}, 1_2)
+1_2\circ (\beta_5, 1_1)+\beta_6$      & $6$ \\ \hline
$V_{(2^2)}$  & $1_2\circ (\zeta_{(2^2)}, 1_2)
+1_2\circ (\beta_5, 1_1)+\beta_6$       & $6$ \\ \hline
$0$      & $\beta_6$ & $6$ \\ \hline  
\end{tabular}
\end{center}
where 
\begin{equation}
\label{E5.5.2}\tag{E5.5.2}
\beta_5:=\tau_{2, 3}\ast (5, 4, 3, 2, 1)
+\tau_5\ast (5, 4, 3, 2, 1)
\end{equation} 
is a generator of $\up{5}(5)$ and 
\begin{equation}
\label{E5.5.3}\tag{E5.5.3}
\beta_6:=\tau_{2, 2, 2}\ast (2, 1, 4, 3, 6, 5)
+\tau_{2, 4}\ast (6, 5, 4, 3, 2, 1)
+\tau_{3, 3}\ast (3, 2, 1, 6, 5, 4)
+\tau_6\ast (6, 5, 4, 3,2, 1)
\end{equation}
is a generator of $\up{6}(6)$. 
\item[(2)]
$\gd(\T_3+\up{5})=6$ and 
\[\alpha=1_2\circ (\tau_3, 1_3)+\sum_{\sigma\in \S_6}\sgn(\sigma)\sigma\]
is a generator of the ideal $\T_3+\up{5}$.
\end{enumerate}
\end{prop}

\begin{proof}
To save space, tedious (but elementary) computations 
are omitted. We only give a proof of (2). 

(2) Clearly, $\alpha\in (\T_3+\up{5})(6)$. Observe that 
\[\tau_3=\alpha\circ (1_1, 1_1, 1_1, 1_0, 1_0, 1_0)\in 
\lr{\alpha}(3), \; {\rm and}\ \sum_{\sigma\in \S_6}
\sgn(\sigma)\sigma\in \lr{\alpha}(6).\]
Since $\up{3}(3)$ can be generated by $\tau_3$ as an 
$\S_3$-module, we know that $\T_3$ can be generated by 
$\tau_3$ as an ideal of $\uas$ and $\T_3\subset \lr{\alpha}$. 
By Proposition \ref{yypro3.3}, each element in the Specht 
basis of $\up{5}(5)$ is of form $\tau_{2, 3}^\sigma$ or 
$\tau_5^\sigma$ for some $\sigma\in \S_5$. Therefore, 
$\up{5}(5)\subset \T_3(5)$. By Remark 5.3, we know 
$\up{5}(6)\subset \T_3(6)+\Bbbk \sum_{\sigma\in \S_6}
\sgn(\sigma)\sigma$ since $\dim \T_3(6)=688$ and 
$\dim (\T_3(6)+\up{5}(6))=689$. Therefore, $\up{5}(6)$ 
can be generated by $\sum_{\sigma\in \S_6}\sgn(\sigma)\sigma$ 
and some element in $\T_3(6)$. By Proposition 
\ref{yypro3.9}, 
we have $\up{5}\subset \lr{\tau_3, \sum_{\sigma\in \S_6}
\sgn(\sigma)\sigma}=\lr{\alpha}$, and therefore 
$\T_3+\up{5}=\lr{\alpha}$. 

On the other hand, we have that
\[(\T_3+\up{5})(3)=\up{3}(3),
\quad (\T_3+\up{5})(4)=\T_3(4), 
\quad {\rm and}
\quad (\T_3+\up{5})(5)=\T_3(5).\]
Observe that $\up{5}(6)\not\subset \T_3(6)$. It follows that 
$\gd(\T_3+\up{5})=6$ since $\T_3+\up{5}$ can be generated by 
$\alpha$.
\end{proof}

\subsection{PI-Algebras having low polynomial growth}
\label{yysec5.2}

In this subsection we give some results on the PI-algebras of 
low grade, some of which were proved in \cite{BYZ, GMZ, OV}.

\begin{prop}
\label{yypro5.6}
Let $A$ be a PI-algebra. For $k=0, 2$ or $3$, the 
grade of $A$ is $k$ if and only if $A$ is a PI-algebra 
associated to $\up{k+1}$, and the grade of $A$ cannot 
be $1$. 
\begin{enumerate}
\item[(1)] 
If $k=0$, then the generating degree of $A$ is $2$, the 
polynomial 
\[f(x_1, x_2)=[x_1, x_2]\] 
is a generating identity of $A$, and the codimension sequence 
is $c_n(A)=1$. 
\item[(2)] 
If $k=2$, then the generating degree of $A$ is $4$, the 
polynomial 
\begin{align*}
f(x_1, x_2, x_3, x_4)
=[x_2, x_1][x_4, x_3]+[x_4, x_3, x_2, x_1]+[x_3, x_2, x_1]x_4
\end{align*}
is a generating identity of $A$, and the codimension sequence 
is $c_n(A)=1+{n\choose 2}$.
\item[(3)] 
If $k=3$, then the generating degree of $A$ is $4$, the polynomial
\[f(x_1, x_2, x_3, x_4)
=[x_2, x_1][x_4, x_3]+[x_4, x_3, x_2, x_1]\]
is a generating identity of $A$, and the codimension 
$c_n(A)=1+{n\choose 2}+2{n\choose 3}$.
\end{enumerate}
\end{prop}

\begin{proof}
By \cite[Theorem 0.7]{BYZ}, $\I_{A}=\up{k+1}$ for $k=0,2,3$. We
translate results in operad $\uas$ to the setting of PI-algebras.

(1) In this case $\I_A=\up{1}=\up{2}$ \cite[Lemma 3.8]{BYZ}.
By Proposition \ref{yypro3.10}, we have $\gd(\up{2})=2$ which forces
that $\up{2}=\lr{\tau}$. Therefore, $\Phi_2(\tau)=x_1x_2-x_2x_1$ is a 
generating identity of $A$ where $\Phi_n$ is defined in \eqref{E3.1.1}.

Next we consider (3).

(3) In this case $\I_A=\up{4}$. By Proposition 
\ref{yypro3.9}, it suffices to show that 
$\zeta_4=\tau_{2, 2}\ast (2, 1, 4, 3)+\tau_4\ast (4, 3, 2, 1)$ is 
a generator of $\up{4}(4)$ as an $\Bbbk\S_4$-module.
By a direct computation, we know that $\tau_{2, 2}\ast (2, 1, 4, 3)$ 
and $\tau_4\ast (4, 3, 2, 1)$ can be generated by $\zeta_4$ since
\begin{align*}
\tau_4=& \frac{1}{4}(\zeta_4+\zeta_4\ast 
(2, 1, 3, 4))\ast((1,4,2,3)-(1,3,2,4)+(2,3,1,4)-(2,4,1,3)-2(3, 4, 1, 2)).
\end{align*}
From the Specht basis of $\up{4}(4)$, it follows that $\up{4}(4)$ 
can be generated by $\zeta_4$ as an $\Bbbk\S_4$-module.

(2) In this case $\I_{A}=\up{3}$. The assertion basically follows 
from the proof Proposition \ref{yypro4.9}(2), but we give some 
details below. Denote $\zeta'_4=\tau_{2, 2}\ast (2, 1, 4, 3)
+\tau_{4}\ast (4, 3, 2, 1)+(1_2\ucr{1}\tau_3)\ast (3, 2, 1, 4)$. 
Observe that $\zeta'_4\ucr{4} 1_0=\tau_3\ast (3, 2, 1)\in 
\lr{\up{3}(4)}(3)$ and $\tau_3\in \lr{\up{3}(4)}(3)$.
Therefore, we have 
\begin{align*}
\tau_4\ast (4, 3, 2, 1)=(\tau_3\ucr{1}\tau_2)\ast (4, 3, 2, 1)
  \in \lr{\up{3}(4)}(4),\\
(1_2\ucr{1}\tau_3)\ast (3, 2, 1, 4) \in \lr{\up{3}(4)}(4).
\end{align*}
since $\lr{\up{3}(4)}$ is an ideal. Moreover, 
\[\tau_{2, 2}\ast (2, 1, 4, 3)=\zeta'_4-\tau_4\ast (4, 3, 2, 1)
-(1_2\ucr{1}\tau_3)\ast (3, 2, 1, 4)\in \lr{\up{3}(4)}(4).\]
It follows that $\lr{\up{3}(4)}(4)=\up{3}(4)$ since 
$\lr{\up{3}(4)}(4)$ contains all of the Specht basis of $\up{4}(4)$.
By Proposition \ref{yypro3.10}, 
we know that $\gd(\up{3})=\gd(\up{4})=4$ and therefore 
$\up{3}=\lr{\zeta'_4}$. 
Consequently, 
\[\Phi_4(\zeta'_4)=[x_2, x_1][x_4, x_3]
+[x_4, x_3, x_2, x_1]+[x_3, x_2, x_1]x_4\]
is a generating identity of $A$.
\end{proof}

By Theorem \ref{yythm5.2} and Proposition \ref{yypro4.9}, we 
have the following facts, and results on the codimension 
sequences and the codimension series were also given in \cite{OV}.

\begin{prop}
\label{yypro5.7}
Let $A$ be a PI-algebra. Then the grade of $A$ is $4$ 
if and only if $A$ is a PI-algebra associated to $\I_A$
as in Theorem {\rm{\ref{yythm5.2}}}. 
\begin{enumerate}
\item[(1)]
If $\I_A=\T_3+\up{5}$, then the codimension sequence is	
\[c_n(A)=1+{n\choose 2}+{n\choose 4},\]
and the codimension series is 
\[c(A,t)=\frac{1}{1-t}+\frac{t^2}{(1-t)^3}+
\frac{t^4}{(1-t)^5}.\]
In this case, $\gd(A)=6$, and the polynomial
\begin{equation}
\label{E5.7.1}\tag{E5.7.1}
f=[x_3, x_1, x_2]x_4x_5x_6+\sum_{\sigma\in \S_6}
\sgn(\sigma)x_{\sigma(1)}\cdots x_{\sigma(6)}
\end{equation}
is a generating identity of $A$.
\item[(2)] 
If $\I_A=\up{4}^M$ with $M$ a proper submodule of $\up{4}(4)$, 
then the codimension sequence is
\[c_n(A)=1+{n\choose 2}+2{n\choose 3}+u{n\choose 4}\]
and the codimension series is 
\[c(A,t)=\frac{1}{1-t}+\frac{t^2}{(1-t)^3}+
\frac{2t^3}{(1-t)^4}+\frac{ut^4}{(1-t)^5},\]
where $u$ is an integer in the set $\{1,2,3,4,5,6,7,8,9\}$.
Generating identities of $A$ are listed as follows:
\begin{equation}
\label{E5.7.2}\tag{E5.7.2}
Table
\end{equation}
~\begin{center}
\begin{tabular}{|c|c|c|c|}
\hline
$M$                               & $u$ & {\rm generating identity}             & $\gd(\up{4}^M)$ \\ \hline 
$V_{(1^4)}+V_{(3,1)}+V_{(2,1^2)}$ & $2$ & $f_1+f_3+f_4$                     & $4$ \\ \hline
$V_{(1^4)}+V_{(2^2)}+V_{(3,1)}$   & $3$ & $f_1+f_2+f_3$                     & $4$ \\ \hline
$V_{(1^4)}+V_{(3,1)}$             & $5$ & $f_1+f_3$                         & $4$ \\ \hline
$V_{(1^4)}+V_{(2^2)}+V_{(2,1^2)}$ & $3$ & $(f_1+f_2+f_4)x_5+g_5$            & $5$ \\ \hline 
$V_{(1^4)}+V_{(2,1^2)}$           & $5$ & $(f_1+f_4)x_5+g_5$                & $5$ \\ \hline
$V_{(1^4)}+V_{(2^2)}$             & $6$ & $(f_1+f_2)x_5+g_5$                & $5$ \\ \hline
$V_{(1^4)}$                       & $8$ & $f_1x_5+g_5$                      & $5$ \\ \hline
$V_{(2^2)}+V_{(3,1)}+V_{(2,1^2)}$ & $1$ &  $(f_2+f_3+f_4)x_5x_6+g_5x_6+g_6$ & $6$ \\ \hline
$V_{(3,1)}+V_{(2,1^2)}$           & $3$ & $(f_3+f_4)x_5x_6+g_5x_6+g_6$      & $6$ \\ \hline
$V_{(2^2)}+V_{(2,1^2)}$           & $4$ &  $(f_2+f_4)x_5x_6+g_5x_6+g_6$     & $6$ \\ \hline
$V_{(2^2)}+V_{(3,1)}$             & $4$ & $(f_2+f_3)x_5x_6+g_5x_6+g_6$      & $6$ \\ \hline
$V_{(3,1)}$                       & $6$ &  $f_3x_5x_6+g_5x_6+g_6$           & $6$ \\ \hline
$V_{(2,1^2)}$                     & $6$ & $f_4x_5x_6+g_5x_6+g_6$            & $6$ \\ \hline
$V_{(2^2)}$                       & $7$ & $f_2x_5x_6+g_5x_6+g_6$            & $6$ \\ \hline
$0$                               & $9$ & $g_6$                             & $6$ \\ \hline
\end{tabular}
\end{center}
where the polynomials
\begin{align*}
f_1=& 2h_1-2h_2+2h_3-h_4+h_5+h_6-h_7-h_8+h_9\\
f_2=& 2h_1+4h_2+2h_3-h_4+h_5-2h_6+2h_7-h_8+h_9\\
f_3=& h_4+3h_5-2h_6+2h_7-3h_8-h_9\\
f_4=& h_4-h_5-h_6+h_7+h_8-h_9\\
g_5=& [x_5, x_4][x_3, x_2, x_1]+[x_5, x_4, x_3, x_2, x_1]\\
g_6=& [x_2, x_1][x_4, x_3][x_6, x_5]+[x_6, x_5][x_4, x_3, x_2, x_1]+[x_3, x_2, x_1][x_6, x_5, x_4]+[x_6, x_5, x_4, x_3, x_2, x_1]
\end{align*}
and 
\begin{align*}
	\begin{split}
		h_1=& [x_2, x_1][x_4, x_3], \quad 
		h_2= [x_3, x_1][x_4, x_2], \quad 
		h_3= [x_3, x_2][x_4, x_1], \\
		h_4=& [x_4, x_1, x_2, x_3], \quad
		h_5= [x_4, x_1, x_3, x_2], \quad 
		h_6= [x_4, x_2, x_1, x_3], \\
		h_7=& [x_4, x_2, x_3, x_1], \quad
		h_8= [x_4, x_3, x_1, x_2], \quad
		h_9= [x_4, x_3, x_2, x_1]. 
	\end{split}
\end{align*}
is the Specht basis of $\Gamma_4$. 
\end{enumerate}
\end{prop}

\begin{proof}
In each case, the assertion follows by applying $\Phi_n$ to
a generator of $\I_A$ given in Subsection \ref{yysec5.1}.
\end{proof}

Theorem \ref{yythm0.9} follows from Propositions 
\ref{yypro5.6} and \ref{yypro5.7}.

\subsection{A partial classification of varieties of grade $5$}
\label{yysec5.3}

The minimal varieties of growth $n^5$ were given in \cite[Section 5]{GMZ}, 
and the varieties of growth $n^5$ such that the associated operadic 
ideal $\I$ with $\dim(\up{5}(5)/(\up{5}(5)\cap \I(5)))=4$ or $44$ were 
obtained in \cite{OV}. In this subsection we provide a partial 
classification of all varieties of growth $n^5$ using operadic 
language. The next lemma is Theorem \ref{yythm0.10}.

\begin{lem}
\label{yylem5.8}
Let $\ip=\uas$. Let $\I$ be an ideal of $\uas$ such that 
$\gkdim \uas/\I=6$. Then one of the following occurs.
\begin{enumerate}
\item[(1)]
$\I=\up{5}^M$ where $M$ is a proper $\Bbbk \S_5$-submodule 
$M\subset \up{5}(5)$.
\item[(2)]
$\I=\upa{5}^{\mathbb M}$ where ${\mathbb M}$ is a $5$-admissible 
pair as defined in Definition {\rm{\ref{yydef4.10}}}.
\end{enumerate}
\end{lem}

\begin{proof} Since $\gkdim \uas/\I=6$, 
$\up{6}\subset \I$ and $\up{5}\not \subset 
\I$ by Corollary \ref{yycor1.6}. By 
Lemma \ref{yylem1.12}, $\mdeg(\I)\leq 6$. 

First we assume that $\mdeg(\I)\leq 3$. Since $\I$ is an 
ideal of $\uas$, it is well-known that $\I\subset \up{2}$.
So $\mdeg(\I)$ is $2$ or $3$. If $\mdeg(\I)=2$, then 
$\tau_{2}\in \I(2)$ and $\lr{\tau_2}=\up{2}$, yielding 
a contradiction. If $\mdeg(\I)=3$, then $\I(3)=\I(3)
\cap \up{3}(3)\neq 0$. Since $\up{3}(3)$ is 
2-dimensional simple, we have $\I(3)=\up{3}(3)$ and
consequently, $\tau_3\in \I(3)$. As a consequence,
$\I\supset \lr{\tau_3}$. By Proposition 
\ref{yypro3.3}, $\up{5}(5)$ is generated by $\tau_3$.
Hence $\up{5}(5)\subset \I$. Hence $\I\supset
\up{5}(5)+\up{6}=\up{5}$. This contradicts with
the hypothesis $\gkdim (\uas/\I)=6$. Therefore $\mdeg(\I)
\neq 3$.

It remains to consider the case when $\mdeg(\I)=4,5,6$.

(1) Suppose $\mdeg(\I)=6$. Then $\I\subset \up{6}$
[Lemma \ref{yylem1.12}], consequently, $\I=\up{6}$
which is also equal to $\up{5}^M$ for $M=0$.
Suppose $\mdeg(\I)=5$. Then the assertion follows from 
Corollary \ref{yycor4.5}(2).

(2) Suppose $\mdeg(\I)=4$. The assertion follows from 
Theorem \ref{yythm4.16}(1).
\end{proof}

For the rest of this section we would like to understand 
all $5$-admissible pairs ${\mathbb M}=(M_4,M_5)$ for the 
operad $\uas$. First of all there are 14 choices of $M_4$ 
as a nonzero proper submodules of $\up{4}(4)$. By definition, 
for each $M_4$, we need to work out all $M_5$ such that 
\begin{equation}
\label{E5.8.1}\tag{E5.8.1}
\langle M_4\rangle \cap \up{5}(5)\subset M_5
\subsetneq \up{5}(5).
\end{equation} 
We have the following Lemmas.

\begin{lem}
\label{yylem5.9}
Let $\mathbb{M}=(M_4, M_5)$ be a 5-admissible pair. Then 
$\gd(\up{4}^{M_4})\ge 5$.
\end{lem}
	
\begin{proof} 
Assume to the contrary that $\gd(\up{4}^{M_4})<5$. Then 
$\up{4}^{M_4}=\lr{M_4}$. Since $\upa{5}^{\mathbb{M}}
\supset \lr{M_4}=\up{4}^{M_4}\supset \up{5}$, we have 
$\gkdim (\uas/\upa{5}^{\mathbb{M}})\le 5$. This is a 
contradiction. Consequently, $\gd(\up{4}^{M_4})\ge 5$. 
\end{proof}

In the setting of Lemma \ref{yylem5.9}, by Proposition 
\ref{yypro5.5}(1), $M_4$ cannot be $V_{(1^4)}+V_{(3,1)}+V_{(2,1^2)}$, 
or $V_{(1^4)}+V_{(2^2)}+V_{(3,1)}$, or $V_{(1^4)}+V_{(3,1)}$ 
for a 5-admissible pair $\mathbb{M}=(M_4, M_5)$. By a direct 
computation, we have the following table.
	
\begin{center}
\begin{tabular}{|c|c|c|c|}
\hline
			$M_4$                             & $\gd(\up{4}^{M_4})$   & The character of $(\langle M_4\rangle\cap \up{5})(5)$ \\ \hline 
			$V_{(1^4)}+V_{(2^2)}+V_{(2,1^2)}$ & 5                     & $2\chi_{(2,1^3)}+2\chi_{(3,2)}+2\chi_{(2^2,1)}+2\chi_{(3,1^2)}$ \\ \hline 
			$V_{(1^4)}+V_{(2,1^2)}$           & 5                     & $2\chi_{(2,1^3)}+2\chi_{(3,2)}+2\chi_{(2^2,1)}+2\chi_{(3,1^2)}$ \\ \hline
			$V_{(1^4)}+V_{(2^2)}$             & 5                     & $2\chi_{(2,1^3)}+2\chi_{(3,2)}+2\chi_{(2^2,1)}+2\chi_{(3,1^2)}$ \\ \hline
			$V_{(1^4)}$                       & 5                     & $2\chi_{(2,1^3)}+\chi_{(2^2,1)}+\chi_{(3,1^2)}$ \\ \hline
			$V_{(2^2)}+V_{(3,1)}+V_{(2,1^2)}$ & 6                     & $\chi_{\up{5}(5)}$ \\ \hline
			$V_{(3,1)}+V_{(2,1^2)}$           & 6                     & $\chi_{\up{5}(5)}$  \\ \hline
			$V_{(2^2)}+V_{(2,1^2)}$           & 6                     & $2\chi_{(2,1^3)}+2\chi_{(3,2)}+2\chi_{(2^2,1)}+2\chi_{(3,1^2)}$ \\ \hline
			$V_{(2^2)}+V_{(3,1)}$             & 6                     & $\chi_{\up{5}(5)}$ \\ \hline
			$V_{(3,1)}$                       & 6                     & $\chi_{(4,1)}+\chi_{(2,1^3)}+2\chi_{(3,2)}+2\chi_{(2^2,1)}+2\chi_{(3,1^2)}$ \\ \hline
			$V_{(2,1^2)}$                     & 6                     & $2\chi_{(2,1^3)}+2\chi_{(3,2)}+2\chi_{(2^2,1)}+2\chi_{(3,1^2)}$ \\ \hline
			$V_{(2^2)}$                       & 6                     & $\chi_{(2,1^3)}+2\chi_{(3,2)}+2\chi_{(2^2,1)}+2\chi_{(3,1^2)}$  \\ \hline
		\end{tabular}
	\end{center}

Therefore, we obtain the following result. 
\begin{lem}
\label{yylem5.10}
Retain the above notation.
\begin{enumerate}
\item[(1)]
If $M$ is $V_{(1^4)}+V_{(3,1)}+V_{(2,1^2)}$, or 
$V_{(1^4)}+V_{(2^2)}+V_{(3,1)}$, or $V_{(1^4)}+V_{(3,1)}$, 
or $V_{(2,2)}+V_{(3,1)}+V_{(2,1^2)}$, 
or $V_{(3,1)}+V_{(2,1^2)}$, or $V_{(2^2)}+V_{(3,1)}$, 
then there is no $5$-admissible pair such that $M_4=M$.
\item[(2)]
If $M$ is  $V_{(1^4)}+V_{(2^2)}+V_{(2,1^2)}$, or 
$V_{(1^4)}+V_{(2,1^2)}$, or $V_{(1^4)}+V_{(2^2)}$, 
or $V_{(2^2)}+V_{(2,1^2)}$, or $V_{(3,1)}$, or 
$V_{(2,1^2)}$, there is exactly one $5$-admissible 
pair such that $M_4=M$.	
\item[(3)]
If $M$ is  $V_{(2^2)}$,  there are  three $5$-admissible 
pairs such that $M_4=M$.
\item[(4)]
If $M$ is  $V_{(1^4)}$,  there are infinitely  $5$-admissible 
pairs such that $M_4=M$.	
\end{enumerate} 		
\end{lem}

Next, we consider the generating degree (and the generators) 
of the operadic ideals of $\uas$ of type $(5, 1)$.

\begin{lem}
\label{yylem5.11}
Let $\I=\up{5}^{M}$ be as in Lemma {\rm{\ref{yylem5.8}(1)}} 
where $M$ is a proper submodule of $\up{5}(5)$. Then 
$\gd(\I)=6$ and a generating element of $\I$ is
$$1_2\circ(\zeta, 1_1)+\beta_6$$
where $\zeta$ is a generator of $M$ as an $\Bbbk \S_5$-module 
and $\beta_6$ is a generator of $\up{6}(6)$ given in 
\eqref{E5.5.3}. 
\end{lem}

\begin{proof}
By Proposition \ref{yypro4.9}, $\gd(\I)$ is either 5 or 6. 
If $\gd(\I)=5$, then $\I=\lr{M}\subset \lr{\up{5}(5)}$,
which does not contain $\up{6}$ by Proposition 
\ref{yypro3.9}.
So $\gkdim (\uas/\I)>6$, yielding a contradiction. Therefore 
$\gd(\I)=6$. By Proposition \ref{yypro4.9}(2), 
$1_2\circ(\zeta, 1_1)+\beta_6$ is a generator of $\I$.
\end{proof}

We are not going to list all submodules $M\subset \up{5}(5)$ and 
the corresponding generators $\zeta$ in Lemma \ref{yylem5.9}.

\begin{lem}
\label{yylem5.12}
Let $\I=\upa{5}^{\mathbb M}$ be as in Lemma 
{\rm{\ref{yylem5.8}(2)}} where ${\mathbb M}=(M_4,M_5)$ is 
a $5$-admissible pair.  If $\gd(\up{4}^{M_4})=6$, then 
$\gd(\upa{5}^{\mathbb M})=6$.
\end{lem}

\begin{proof}
 Suppose to the contrary that $\gd(\upa{5}^{\mathbb M})<6$.
Then $\I=\lr{\I(5)}$. Note that 
$$\I+\lr{\up{5}(5)}=
\lr{M_4}+\lr{M_5}+\up{6}+\lr{\up{5}(5)}=\lr{M_4}+\lr{M_5}+\up{5}
=\lr{M_4}+\up{5}=\up{4}^{M_4}.$$
Since both $\I$ and $\lr{\up{5}(5)}$ are generated in degree
$5$, so is $\up{4}^{M_4}$. This contradicts the hypothesis.
\end{proof}

From Lemmas	\ref{yylem5.9} and	\ref{yylem5.12}, it remains  
to consider the case of $\gd(\up{4}^{M_4})=5$.

\begin{lem}
\label{yylem5.13}
Retain the above notation.
\begin{enumerate}
\item[(1)]
If $M_4$ is $V_{(1^4)}+V_{(2^2)}+V_{(2,1^2)}$, or 
$V_{(1^4)}+V_{(2,1^2)}$, or $V_{(1^4)}+V_{(2^2)}$, and  
${\mathbb M}=(M_4,M_5)$ be the corresponding $5$-admissible 
pair. Then $\gd(\upa{5}^{\mathbb M})=6$.
\item[(2)]
If $M_4$ is  $V_{(1^4)}$, and  $M_5$ does not contain the 
irreducible representation corresponding to partition 
$(4,1)$, then $\gd(\upa{5}^{\mathbb M})=6$.
\end{enumerate} 		
\end{lem}	

\begin{proof} To save some space we only give some details
in part (1).

We only need to consider the case 
$M_4=V_{(1^4)}+V_{(2^2)}+V_{(2,1^2)}$, since it corresponds 
to the largest generalized truncation ideals of type II 
among above two parts. Through computation assisted by 
computer, we get
\[\chi_{\lr{M_4}(6)}=\chi_{(1^6)}
+5\chi_{(2, 1^4)}+4\chi_{(3^2)}+5\chi_{(2^3)}
+6\chi_{(4, 2)}+9\chi_{(2^2, 1^2)}+10\chi_{(3, 1^3)}
+6\chi_{(4, 1^2)}+14\chi_{(3, 2, 1)}.\]
Note that it follows from a direct 
computation,
\[\chi_{\up{6}(6)}=\chi_{(1^6)}+\chi_{(5, 1)}
+2\chi_{(2, 1^4)}+2\chi_{(3^2)}+2\chi_{(2^3)}
+3\chi_{(4, 2)}+4\chi_{(2^2, 1^2)}+4\chi_{(3, 1^3)}
+3\chi_{(4, 1^2)}+6\chi_{(3, 2, 1)}\]
(which is also given in \cite{JSV}). 
So we obtain that $\up{6}(6)\not\subset \lr{M_4}(6)$, and 
consequently, $\gd(\upa{5}^{\mathbb M})=6$. 	
\end{proof}

\begin{rmk}
\label{yyrem5.14}
It remains to consider the case when $M_4$ is $V_{(1^4)}$ and  
$M_5$  contains the irreducible representation corresponding 
to the partition $(4,1)$. In the first table below, the module 
$M_5$ is uniquely determined, and we  can calculate the 
corresponding generating degrees.		
\begin{center}	
\begin{tabular}{|c|c|c|}
\hline
The character of $ M_5$                                                       
&$\gd(\upa{5}^{\mathbb M})$ \\ 
\hline 
$\chi_{(4,1)}+2\chi_{(2,1^3)}+\chi_{(2^2,1)}+\chi_{(3,1^2)}$                  
& $6$ \\ 
\hline			 
$\chi_{(4,1)}+2\chi_{(2,1^3)}+\chi_{(2^2,1)}+\chi_{(3,1^2)}+2\chi_{(3,2)}$    
& $6$ \\ 
\hline
$\chi_{(4,1)}+2\chi_{(2,1^3)}+\chi_{(2^2,1)}+2\chi_{(3,1^2)}$                 
& $6$ \\ 
\hline
$\chi_{(4,1)}+2\chi_{(2,1^3)}+\chi_{(2^2,1)}+2\chi_{(3,1^2)}+2\chi_{(3,2)}$   
& $5$\\ 
\hline
$\chi_{(4,1)}+2\chi_{(2,1^3)}+2\chi_{(2^2,1)}+\chi_{(3,1^2)}$                 
& $6$ \\ 
\hline
$\chi_{(4,1)}+2\chi_{(2,1^3)}+2\chi_{(2^2,1)}+\chi_{(3,1^2)}+2\chi_{(3,2)}$   
& $5$ \\ 
\hline
$\chi_{(4,1)}+2\chi_{(2,1^3)}+2\chi_{(2^2,1)}+2\chi_{(3,1^2)}$                
& $5$  \\ 
\hline		
\end{tabular}
\end{center}
\end{rmk}

In the second table,  the module $M_5$ has infinitely many choices 
and it is difficult to determine the corresponding generating 
degrees. But when the character of $M_5$ is 
$\chi_{(4,1)}+2\chi_{(2,1^3)}+\chi_{(2^2,1)}+\chi_{(3,1^2)}
+\chi_{(3,2)}$ or 
$\chi_{(4,1)}+2\chi_{(2,1^3)}+2\chi_{(2^2,1)}+2\chi_{(3,1^2)}
+\chi_{(3,2)}$, from the preceding table, we can determine 
their generating degrees.

\begin{center}	
\begin{tabular}{|c|c|c|}
\hline
The character of $ M_5$                                                       
&$\gd(\upa{5}^{\mathbb M})$ \\ 
\hline 
$\chi_{(4,1)}+2\chi_{(2,1^3)}+\chi_{(2^2,1)}+\chi_{(3,1^2)}+\chi_{(3,2)}$     
& $6$ \\ 
\hline
$\chi_{(4,1)}+2\chi_{(2,1^3)}+\chi_{(2^2,1)}+2\chi_{(3,1^2)}+\chi_{(3,2)}$    
& $5$ or $6$ \\ 
\hline
$\chi_{(4,1)}+2\chi_{(2,1^3)}+2\chi_{(2^2,1)}+\chi_{(3,1^2)}+\chi_{(3,2)}$    
& $5$ or $6$ \\ 
\hline
$\chi_{(4,1)}+2\chi_{(2,1^3)}+2\chi_{(2^2,1)}+2\chi_{(3,1^2)}+\chi_{(3,2)}$   
& $5$ \\ 
\hline			
\end{tabular}
\end{center}

Combining Lemmas \ref{yylem5.11}, \ref{yylem5.12}, 
and \ref{yylem5.13} with Remark \ref{yyrem5.14}, we have 
the following conclusion on the generating degree for 
generalized truncation ideals $\upa{5}^{\mathbb M}$ for
all possible $5$-admissible pairs ${\mathbb M}$.
\begin{equation}
\label{E5.14.1}\tag{E5.14.1}
{\text{Table of generating degree of $\upa{5}^{\mathbb M}$ in terms of
${\mathbb M}$}}
\end{equation}
~\begin{center}
\begin{tabular}{|c|c|c|c|}
\hline
		$M_4$                               & The number of $M_5$                & The character of $M_5$                                                      & $\gd(\upa{5}^{\mathbb M})$ \\ \hline
		$V_{(2^2)}+V_{(2,1^2)}$             & 1                                  & $2\chi_{(2,1^3)}+2\chi_{(3,2)}+2\chi_{(2^2,1)}+2\chi_{(3,1^2)}$             & 6     \\ \hline
		$V_{(3,1)}$                         & 1                                  & $\chi_{(4,1)}+\chi_{(2,1^3)}+2\chi_{(3,2)}+2\chi_{(2^2,1)}+2\chi_{(3,1^2)}$ & 6     \\ \hline
		$V_{(2,1^2)}$                       & 1                                  & $2\chi_{(2,1^3)}+2\chi_{(3,2)}+2\chi_{(2^2,1)}+2\chi_{(3,1^2)}$             & 6     \\ \hline
		\multirow{3}{*}{$V_{(2^2)}$}        & \multirow{3}{*}{3}                 & $\chi_{(2,1^3)}+2\chi_{(3,2)}+2\chi_{(2^2,1)}+2\chi_{(3,1^2)}$             & 6     \\
		\cline{3-4}
		&                                    & $2\chi_{(2,1^3)}+2\chi_{(3,2)}+2\chi_{(2^2,1)}+2\chi_{(3,1^2)}$ & 6     \\ 
		\cline{3-4}
		&                                    & $\chi_{(4,1)}+\chi_{(2,1^3)}+2\chi_{(3,2)}+2\chi_{(2^2,1)}+2\chi_{(3,1^2)}$ & 6     \\ \hline
		$V_{(1^4)}+V_{(2^2)}+V_{(2,1^2)}$   & 1                                  & $2\chi_{(2,1^3)}+2\chi_{(3,2)}+2\chi_{(2^2,1)}+2\chi_{(3,1^2)}$             & 6     \\ \hline
		$V_{(1^4)}+V_{(2,1^2)}$             & 1                                  & $2\chi_{(2,1^3)}+2\chi_{(3,2)}+2\chi_{(2^2,1)}+2\chi_{(3,1^2)}$             & 6     \\ \hline
		$V_{(1^4)}+V_{(2^2)}$               & 1                                  & $2\chi_{(2,1^3)}+2\chi_{(3,2)}+2\chi_{(2^2,1)}+2\chi_{(3,1^2)}$             & 6     \\ \hline
		\multirow{23}{*}{$V_{(1^4)}$}       & \multirow{23}{*}{Infinitely many } & $2\chi_{(2,1^3)}+\chi_{(2^2,1)}+\chi_{(3,1^2)}$                             & 6       \\
		\cline{3-4}
		&                                    & $2\chi_{(2,1^3)}+\chi_{(2^2,1)}+2\chi_{(3,1^2)}$                              & 6       \\
		\cline{3-4}
		&                                    & $2\chi_{(2,1^3)}+2\chi_{(2^2,1)}+\chi_{(3,1^2)}$                              & 6       \\
		\cline{3-4}
		&                                    & $2\chi_{(2,1^3)}+2\chi_{(2^2,1)}+2\chi_{(3,1^2)}$                             & 6       \\
		\cline{3-4}
		&                                    & $2\chi_{(2,1^3)}+\chi_{(2^2,1)}+\chi_{(3,1^2)}+2\chi_{(3,2)}$               & 6       \\
		\cline{3-4}
		&                                    & $2\chi_{(2,1^3)}+\chi_{(2^2,1)}+2\chi_{(3,1^2)}+2\chi_{(3,2)}$              & 6       \\
		\cline{3-4}
		&                                    & $2\chi_{(2,1^3)}+2\chi_{(2^2,1)}+\chi_{(3,1^2)}+2\chi_{(3,2)}$                & 6       \\
		\cline{3-4}
		&                                    & $2\chi_{(2,1^3)}+2\chi_{(2^2,1)}+2\chi_{(3,1^2)}+2\chi_{(3,2)}$               & 6       \\
		\cline{3-4}
		&                                    & $2\chi_{(2,1^3)}+\chi_{(2^2,1)}+\chi_{(3,1^2)}+\chi_{(3,2)}$                & 6       \\
		\cline{3-4}
		&                                    & $2\chi_{(2,1^3)}+\chi_{(2^2,1)}+2\chi_{(3,1^2)}+\chi_{(3,2)}$               & 6       \\
		\cline{3-4}
		&                                    & $2\chi_{(2,1^3)}+2\chi_{(2^2,1)}+\chi_{(3,1^2)}+\chi_{(3,2)}$                 & 6       \\
		\cline{3-4}
		&                                    & $2\chi_{(2,1^3)}+2\chi_{(2^2,1)}+2\chi_{(3,1^2)}+\chi_{(3,2)}$                & 6       \\
		\cline{3-4}
		&                                    & $\chi_{(4,1)}+2\chi_{(2,1^3)}+\chi_{(2^2,1)}+\chi_{(3,1^2)}$                & 6       \\
		\cline{3-4}
		&                                    & $\chi_{(4,1)}+2\chi_{(2,1^3)}+\chi_{(2^2,1)}+2\chi_{(3,1^2)}$               & 6       \\
		\cline{3-4}
		&                                    & $\chi_{(4,1)}+2\chi_{(2,1^3)}+2\chi_{(2^2,1)}+\chi_{(3,1^2)}$               & 6       \\
		\cline{3-4}
		&                                    & $\chi_{(4,1)}+2\chi_{(2,1^3)}+2\chi_{(2^2,1)}+2\chi_{(3,1^2)}$                & 5       \\ 									
		\cline{3-4}
		&                                    & $\chi_{(4,1)}+2\chi_{(2,1^3)}+\chi_{(2^2,1)}+\chi_{(3,1^2)}+2\chi_{(3,2)}$  & 6       \\
		\cline{3-4}
		&                                    & $\chi_{(4,1)}+2\chi_{(2,1^3)}+\chi_{(2^2,1)}+2\chi_{(3,1^2)}+2\chi_{(3,2)}$ & 5       \\
		\cline{3-4}
		&                                    & $\chi_{(4,1)}+2\chi_{(2,1^3)}+2\chi_{(2^2,1)}+\chi_{(3,1^2)}+2\chi_{(3,2)}$ & 5       \\									
		\cline{3-4}
		&                                    & $\chi_{(4,1)}+2\chi_{(2,1^3)}+\chi_{(2^2,1)}+\chi_{(3,1^2)}+\chi_{(3,2)}$   & 6       \\
		\cline{3-4}
		&                                    & $\chi_{(4,1)}+2\chi_{(2,1^3)}+\chi_{(2^2,1)}+2\chi_{(3,1^2)}+\chi_{(3,2)}$  & 5 or 6  \\
		\cline{3-4}
		&                                    & $\chi_{(4,1)}+2\chi_{(2,1^3)}+2\chi_{(2^2,1)}+\chi_{(3,1^2)}+\chi_{(3,2)}$  & 5 or 6  \\
		\cline{3-4}
		&                                    & $\chi_{(4,1)}+2\chi_{(2,1^3)}+2\chi_{(2^2,1)}+2\chi_{(3,1^2)}+\chi_{(3,2)}$   & 5       \\ \hline
	\end{tabular}	
\end{center}

Except for two cases, $\gd(\upa{5}^{\mathbb M})$ is 
completely determined. Note that if $M_5$ is unique, then 
$M_5=\lr{M_4}\cap \up{5}(5)\subsetneq \up{5}(5)$. Otherwise 
$\lr{M_4}\cap \up{5}(5)\subseteq M_5\subsetneq \up{5}(5)$
(see \eqref{E5.8.1}).

Finally we give a complete list of the codimension series 
of type $(5, 1)$. 
	
\begin{thm}
\label{yythm5.15}
There are exactly 55 distinct codimension series of type 
$(5,1)$ as follows.
\begin{enumerate}
\item[(1)]
If $\I_A$ is of the form $\up{5}^{M}$ where $M$ is a proper 
submodule of $\up{5}(5)$, we have the following sublist:
$$c(A,t)=\frac{1}{1-t}+\frac{t^2}{(1-t)^3}+\frac{2t^3}{(1-t)^4}
+\frac{9t^4}{(1-t)^5}+\frac{ut^5}{(1-t)^6}$$
or
\[c_n(A)= 1+{n\choose 2}+2{n\choose 3}+9{n\choose 4}+u{n\choose 5}\]
where $u$ is an integer in the set 
$[44] \backslash \{1,2,3,7,37,41,42,43\}$.
\item[(2)]
If $\I_A$ is of the form $\up{5}^{\mathbb M}$ where ${\mathbb M}$ 
is a $5$-admissible pair, we have the following sublist:
\begin{enumerate}
\item[(2a)]
$$c(A,t)=\frac{1}{1-t}+\frac{t^2}{(1-t)^3}+
\frac{2t^3}{(1-t)^4}+\frac{8t^4}{(1-t)^5}+\frac{xt^5}{(1-t)^6}$$
or
\[c_n(A)= 1+{n\choose 2}+2{n\choose 3}+8{n\choose 4}+x{n\choose 5}\]
where $x$ is an integer in the set 
$\{ 4,5,6,9,10,11,14,15,16,19,20,21,25 \}$.
\item[(2b)]
$$c(A,t)=\frac{1}{1-t}+\frac{t^2}{(1-t)^3}+\frac{2t^3}{(1-t)^4}
+\frac{7t^4}{(1-t)^5}+\frac{yt^5}{(1-t)^6}$$
or
\[c_n(A)= 1+{n\choose 2}+2{n\choose 3}+7{n\choose 4}+y{n\choose 5}\]
where $y$ is an integer in the set $\{ 4,8 \}$.
\item[(2c)]
$$c(A,t)=\frac{1}{1-t}+\frac{t^2}{(1-t)^3}+\frac{2t^3}{(1-t)^4}
+\frac{zt^4}{(1-t)^5}+\frac{4t^5}{(1-t)^6}$$
or
\[c_n(A)= 1+{n\choose 2}+2{n\choose 3}+z{n\choose 4}+4{n\choose 5}\]
where $z$ is an integer in the set $\{ 3,4,5,6 \}$.
\end{enumerate}
\end{enumerate}
As a consequence, 
$\Lambda_{(5,1)}=\{\frac{u}{5!}\mid  u\in [44] 
\backslash \{1,2,3,7,37,41,42,43\}\}$.
\end{thm}

\begin{proof} By Lemma \ref{yylem5.8}, there
are two cases to consider: (1) $\I_A$ is $\up{5}^M$ 
where $M$ is a proper submodule of $\up{5}(5)$, or
(2) $\I_A=\upa{5}^{\mathbb M}$ where ${\mathbb M}$ is 
a $5$-admissible pair. In the first case, the 
codimension series appear as in part (1) by
Corollary \ref{yycor4.6}. In the second case,
the codimension series appear as in parts (2a,2b,2c)
by Lemma \ref{yylem4.14} and Table \eqref{E5.14.1}.
For each $(M_4,M_5)$ in the Table \eqref{E5.14.1},
the computation of the codimension series is 
straightforward by Lemma \ref{yylem4.14}, so 
details are omitted. 
\end{proof}

\noindent\textbf{Acknowledgments}.
We would like to thank Xiao-Wu Chen and Ji-Wei He for many 
useful suggestions. Y.-H. Bao was supported by NSFC (Grant
No. 12371015),  the Science Fundation for Distinguished 
Young Scholars of Anhui Province (No. 2108085J01) and 
Excellent University Research and Innovation Team in 
Anhui Province (No. 2024AH010002). Y. Ye was partially 
supported by the National Key R\&D Program
of China (No. 2024YFA1013802), the National Natural Science Foundation of 
China (Nos. 12131015 and 12371042) 
and the Innovation Program for Quantum Science and 
Technology (No. 2021ZD0302902). 
J.J. Zhang was partially supported by the US National
Science Foundation (Nos. DMS-2001015 and DMS-2302087).

\end{document}